\documentclass[12pt,a4paper,reqno]{amsart}
\usepackage{stmaryrd}
\usepackage{comment}

\usepackage{amsaddr}

\usepackage{datetime2}
\usepackage{hyperref}

\usepackage{graphicx}
\usepackage{tikz}
\usetikzlibrary{decorations.fractals}

\usepackage{graphicx,xcolor}

\usepackage{amsmath,amsthm}
\usepackage{amssymb,latexsym}
\usepackage{enumerate}
\usepackage{amscd,amssymb,amsthm}
\usepackage{graphicx}

\usepackage{mathrsfs,amssymb}

\usepackage{amscd,amssymb,amsthm}
\usepackage{graphicx}

\usepackage{fancyhdr}

\usepackage{amsmath,amsthm}
\usepackage{amssymb,latexsym}
\usepackage{enumerate}
\usepackage{amscd,amssymb,amsthm}
\usepackage{graphicx}
\usepackage{esint}

\usepackage{xcolor}

\usepackage{textcomp,stmaryrd}

\usepackage{appendix}

\allowdisplaybreaks

\newtheorem{thm}{Theorem}[section]
\newtheorem{prop}[thm]{Proposition}
\newtheorem{cor}[thm]{Corollary}
\newtheorem{lem}[thm]{Lemma}
\theoremstyle{remark}
\newtheorem{rmk}[thm]{Remark}
\theoremstyle{definition}
\newtheorem{defn}{Definition}[section]

\DeclareMathOperator{\N}{\mathbb{N}}

\DeclareMathOperator{\R}{\mathbb{R}}
\DeclareMathOperator{\cF}{\mathcal{F}}

\DeclareMathOperator{\cA}{\mathcal{A}}
\DeclareMathOperator{\cB}{\mathcal{B}}
\DeclareMathOperator{\cC}{\mathcal{C}}

\DeclareMathOperator{\cM}{\mathcal{M}}

\DeclareMathOperator{\cU}{\mathcal{U}}

\DeclareMathOperator{\fF}{\mathfrak{F}}

\newcommand{\UUU}{\color{black}}
\newcommand{\EEE}{\color{black}}
\newcommand{\RRR}{\color{black}}
\definecolor{MyDarkGreen}{rgb}{0,0.50,0.04}
\newcommand{\GGG}{\color{black}}

\makeatletter
\@namedef{subjclassname@2020}{%
  \textup{2020} Mathematics Subject Classification}
\makeatother

\begin{document}
\title[Regularity for double-phase orthotropic functionals]{Gradient regularity for double-phase \\ orthotropic functionals}

\author[S. Almi]{Stefano Almi}
\address{Dipartimento di Matematica e Applicazioni ``R. Caccioppoli",
	Universit\`a degli Studi di Napoli  Federico II,
	Via Cintia, 80126 Napoli, Italy}
\email{stefano.almi@unina.it}

\author[C. Leone]{Chiara Leone}
\address{Dipartimento di Matematica e Applicazioni ``R. Caccioppoli",
	Universit\`a degli Studi di Napoli  Federico II,
	Via Cintia, 80126 Napoli, Italy}
	\email{chiara.leone@unina.it}

\author[G. Manzo]{Gianluigi Manzo}
\address{Dipartimento di Matematica e Applicazioni ``R. Caccioppoli",
	Universit\`a degli Studi di Napoli  Federico II,
	Via Cintia, 80126 Napoli, Italy}
\email{gianluigi.manzo@unina.it}


\begin{abstract}
We prove higher integrability for local minimizers of the  double-phase orthotropic functional
\begin{displaymath}
      \sum_{i=1}^{n}\int_\Omega\left(\left|u_{x_i}\right|^p+a(x)\left| u_{x_i}\right|^q\right)dx
\end{displaymath}
when the weight function $a \geq0$ is assumed to be $\alpha$-H\"older continuous, while the exponents $p, q$ are such that $2 \leq p \leq q$ and $\frac{q}{p} < 1 + \frac{\alpha}{n}$. Under natural Sobolev regularity of~$a$, we further obtain explicit Lipschitz regularity estimates for local minimizers.
\end{abstract}

\maketitle

\section{Introduction}
The present paper deals with the gradient regularity for local minimizers of the functional
\begin{equation}\label{funct:main}
	\mathcal{F}_0[u,\Omega]=\sum_{i=1}^{n}\int_\Omega\left(\left|u_{x_i}\right|^p+a(x)\left| u_{x_i}\right|^q\right)dx
\end{equation}
for $u\in W^{1,p}(\Omega)$, \GGG $\Omega\subset\mathbb{R}^n$ a bounded open set, $2\le p\le q$, \EEE and $a \colon {\Omega}\to [0,+\infty)$ a suitable continuous function. The energy~\eqref{funct:main} features both an {\em orthotropic structure} and a {\em double phase structure} and couples in a nontrivial way the difficulties coming from the two situations. \EEE
In the simplest case, an orthotropic functional is of the form
\begin{equation}
\label{e:intro1}
	\int_\Omega G(\nabla u) \,dx \qquad \text{with } G(z) := \frac{1}{p} \sum_{i=1}^{n} | z_{i} |^{p}\,.
\end{equation}
This kind of functional first appeared in the framework of Optimal Transport with congestion effects~\cite{MR3296496, MR2651987}. Key feature of~\eqref{e:intro1} is the lack of ellipticity. By a direct computation we indeed get that
\begin{displaymath}
    \left\langle D^{2} G(z) \xi , \xi \right\rangle =  (p-1)  \sum_{i=1}^{n}| z_{i} |^{p-2} | \xi_{i}|^{2} \qquad \text{for every $z, \xi \in \R^{n}$.}
\end{displaymath}
Hence, $G$ presents a degenerate behavior whenever one component of $z$ is zero, which happens on the union of $n$ hyperplanes. For this reason, the regularity techniques for elliptic integrands are not applicable in the orthotropic setting, and the higher regularity for local minimizers of~\eqref{e:intro1} is not fully understood. For example, the $C^{1}$ regularity has been tackled only in dimension $n=2$ in~\cite{MR4239447, MR3749368, MR3886595, MR3973558}. In order to avoid restrictions on the dimension, we focus here on Lipschitz regularity for local minimizers of the orthotropic double-phase functional~\eqref{funct:main}, which has been studied in the pure orthotropic framework of~\eqref{e:intro1} in~\cite{MR4163990, MR3616334, BouBraLeo, bblv, MR725829}. We further refer to~\cite{MR3461298} for an alternative approach based on viscosity solutions and to~\cite{BrLePiVe16, cavaliere1998, MR4927646} for higher differentiability results.

The second feature of the functional $\mathcal{F}_0$ is a double phase structure. The model case of a double phase energy is
\begin{equation}
\label{e:intro2}
	\int_\Omega |\nabla u|^p+a(x)|\nabla u|^q\,dx
\end{equation}
with $1<p<q<\infty$ and $a(x)\ge 0$ almost everywhere. This functional pertains to the class of variational problems with nonstandard growth conditions, following the terminology of Marcellini in~\cite{MR1094446, MR969900}. These energies have been introduced by Zhikov in~\cite{MR681795, MR864171, zhikov, MR1486765} for modeling composite materials characterized by a point-dependent anisotropic behavior, ruled by the weight $a(\cdot)$ which may change  point-by-point the rate of ellipticity of the integrand.  

To obtain meaningful regularity results, suitable assumptions on $p,q$ and $a(\cdot)$ are needed: in particular, $p$ and $q$ should not be too far apart, in order to avoid the Lavrentiev phenomenon~\cite{borowski2022absence, EspLeoMingio, zhikov}, while $a$ should be sufficiently regular. It was first shown in~\cite{colombo2015regularity} that $\nabla u$ is locally H\"older continuous if the weight~$a$ is in $C^{0, \alpha} ({\Omega})$ for some $\alpha \in (0, 1]$ and $\frac{q}{p} < 1 + \frac{\alpha}{n}$. This result has been then improved to include the borderline case in~\cite{MR3775180, MR4331020, MR4397041}. We refer the interested reader to the recent survey~\cite{MR4258810} and references therein concerning recent advances in regularity theory for minimizers of functionals under nonstandard growth.

In the attempt of understanding the role played by non-autonomous coefficients in the regularity theory for orthotropic functionals, it is rather natural to start from the double-phase type of energy introduced in~\eqref{funct:main}.

Our first result is a higher integrability statement, which is valid assuming  H\"older regularity for $a\colon {\Omega}\to[0,+\infty):$
\begin{equation}\label{hpholder}
a\in C^{0,\alpha}(\Omega),\ \ \ \alpha\in(0,1],
\end{equation}
as well as the bound $\frac{q}{p}< 1+\frac{\alpha}{n}$.  Before stating the result, we briefly recall the definition of local minimizers for $\mathcal{F}_0$.

\begin{defn}
A function $u\in W^{1,p}(\Omega)$ is a local minimizer of the functional defined in \eqref{funct:main} if $\mathcal{F}_0[u,\Omega]<+\infty$ and 
$$
\mathcal{F}_0[u,\Omega']\le \mathcal{F}_0[\varphi,\Omega'] \ \ \ \ \hbox{ for every }u-\varphi\in W^{1,p}_0(\Omega') \hbox{ and every }\Omega'\Subset\Omega.
$$
\end{defn}

\begin{thm}[Higher integrability]
\label{revhol}
	Let $\Omega\subset\R^n$ be a bounded open set, the nonnegative function $a(\cdot)$ satisfy \eqref{hpholder}, and $q\ge p\ge 2$ such that $\displaystyle\frac{q}{p}<1+\frac{\alpha}{n}$.  Let $U\in W^{1,p}(\Omega)$ be a local minimizer of \eqref{funct:main}. Then, for every $\tilde{q}<np/(n-2\alpha)$ there exists an exponent $\tilde{b}=\tilde{b}(n,p,q,\alpha)$ and a constant $c=c(n,p,q,a,\tilde{q},\textrm{diam}(\Omega))$ such that for every ball $B_R\Subset \Omega$ it holds 
	\begin{equation}\label{eq:rh2}
		\left(\fint_{B_{R/2}}|\nabla U|_p^{\tilde{q}}\,dx\right)^{1/\tilde{q}}\le c\left(\fint_{B_R}(|\nabla U|_p^p+1)\,dx\right)^{\tilde{b}/p}.
	\end{equation}
	In particular, \eqref{eq:rh2} 
	holds for $\tilde{q}=q$ and $U\in W^{1,q}_{loc}(\Omega)$.
\end{thm}

\GGG

At this stage, the orthotropic effect is rather mild, as higher integrability is obtained by a refined difference quotient technique (cf.~\cite[Theorem 1.5]{colombo2015regularity}) which only relies on the Euler-Lagrange equation of the functional, and thus is not influenced by the degenerate behavior of~\eqref{funct:main}. The key step is to interpret the H\"older continuity of $a(\cdot)$ in terms of fractional differentiability, which then allows the use of embeddings in Nikol'ski\u{\i} spaces. We refer to Section~\ref{s:higher-integrability} for the details of the proof.

\smallskip
\EEE

To guarantee Lipschitz continuity of local minimizers we are going to assume  that $a(\cdot)$ belongs to a Sobolev function with a suitable high degree of integrability of $\nabla a$, as it was done for instance in~\cite{MR4078712, MR4331020, MR3537963, MR4116617}. More precisely, we consider
\begin{equation}\label{hpsobolev}
 a\ge 0,\ a\in \left\{\begin{array}{ll}
 W^{1,\frac{n}{1-\alpha}}(\Omega) &\hbox{ if } \alpha\in (0,1)\\
 W^{1,\infty}(\Omega) &\hbox{ if } \alpha=1.
 \end{array}
 \right.
\end{equation}
Let us note that if $\Omega$ has Lipschitz boundary the embedding $W^{1,\frac{n}{1-\alpha}}(\Omega)\hookrightarrow C^{0,\alpha}({\overline\Omega})$ holds.

\begin{thm}[Lipschitz regularity]
\label{thm:main}
	Let $\Omega\subset\R^n$ be a bounded open set, the function $a(\cdot)$ satisfy~\eqref{hpsobolev}, and $q\ge p\ge 2$ such that $\displaystyle\frac{q}{p}<1+\frac{\alpha}{n}.$ Then a local minimizer $U\in W^{1,p}(\Omega)$ of \eqref{funct:main} is locally Lipschitz. In particular, there exist a constant $C$ and an exponent $\Theta$ such that for every $B_{R/3}\subset B_R\Subset \Omega$ the estimate
	\begin{equation}
		\|\nabla U\|_{L^\infty(B_{R/3})}\le C \left(\int_{B_R}|\nabla U|^p+1\right)^{\Theta/p}
	\end{equation}
	holds. The constant $C$ depends on $n,p,q,a,R$ while $\Theta$ on $n,p,q,\alpha$.
\end{thm}

 As usual when dealing with higher order regularity, the initial difficulty lies in the fact that the minimizer $U$ lacks the smoothness required to carry out all necessary computations directly. To overcome this issue, in Section~\ref{s:approximation} we approximate the local minimizer $U$ by solutions~$u_\varepsilon$ of uniformly elliptic problems. We will prove in a direct way the absence of Lavrentiev phenomenon implied by $\frac{q}{p}< 1+\frac{\alpha}{n}$. This ensures that the approximating solutions $u_\varepsilon$ are as smooth as needed and converge to the original minimizer $U$. Therefore, it suffices to establish suitable a priori estimates for $u_\varepsilon$ that remain stable as $\varepsilon$ goes to $0$.

The proof strategy builds upon the following steps: {\em (i)} differentiation of the Euler-Lagrange equation associated with the approximating problems; {\em (ii)} derivation of Caccioppoli-type inequalities for convex powers of the gradient components (see Section~\ref{s:caccioppoli}); {\em (iii)} construction of an iterative scheme based on reverse H\"older inequalities (cf.~Propositions~\ref{prop:cac} and~\ref{p:2.2} and Corollary~\ref{c:powers}); and finally {\em (iv)} iteration to obtain the desired local $L^\infty$ estimate on $\nabla u_\varepsilon$ (see Proposition~\ref{p:iteration}). Despite appearing to be a rather standard work plan in regularity theory, the implementation of the above steps is quite delicate due to the degeneracy of the equation under consideration. In fact, we adopt the same general approach as in \cite[Theorem 1.1]{bblv}, but the adaptation is far from straightforward.
Our method ultimately yields an $L^\infty-L^m$ estimate for $\nabla u_\varepsilon$, achieved after an infinite iteration process. One must carefully address a subtle issue characteristic of nonstandard growth problems: the exponent $m$ in this preliminary estimate may turn out to be excessively large. Fortunately, this estimate can be “corrected” using again the reverse H\"older inequality together with an interpolation argument that effectively reduces the initial integrability requirement on $\nabla u_\varepsilon$. This last step is contained in Proposition~\ref{prop:lip}.

\medskip

\noindent{\bf Outlook.} In this work we have considered the restrictions $2 \leq p \leq q$ and $\frac{q}{p} < 1 + \frac{\alpha}{n}$. Future investigations will focus on the full range of exponents $1 < p \leq q$. Let us point out that in the pure orthotropic framework, Lipschitz regularity with subquadratic growth was considered in~\cite{BouBraLeo}. A further issue is the borderline case $\frac{q}{p} = 1 + \frac{\alpha}{n}$ (see~\cite{MR3775180, MR4331020} in the double phase setting), where we lack the embedding in Nikol'ski\u{\i} spaces and the final interpolation argument. Finally, in dimension $n=2$ it would be of interest to extend the $C^{1}$-regularity results of~\cite{MR4239447, MR3749368, MR3886595, MR3973558} to the present setting.

\EEE
\section{Preliminaries}

\GGG In this Section we fix a bounded open set $\Omega\subset \R^n$, $n\ge 2$. For $1\le\gamma<\infty$ and $z\in \R^n$ denote $|z|_\gamma=\left(\sum_{i=1}^{k}|z_i|^\gamma\right)^{1/\gamma}$ and \EEE consider the functional
\begin{equation}\label{eq:main}
	\UUU \cF_{0} \EEE [u;\omega]=\int_\omega \big( |\nabla u(x)|_p^p+a(x)|\nabla u(x)|_q^q \big) \,dx,
\end{equation}
for $u\in W^{1,p}(\Omega)$ and $\omega\subseteq\Omega$ open. We also denote $\cF_{0}[u]=\cF_{0}[u,\Omega].$

The inequality
\begin{equation}\label{eq:pequiv}
	\sum_{i=1}^{m}|x_i|^\gamma\le \left(\sum_{i=1}^m|x_i|\right)^\gamma\le m^{\gamma-1} \sum_{i=1}^{m}|x_i|^\gamma\quad\forall x_1,\dots,x_m\in \mathbb{R}^k
\end{equation}
holds for $\gamma\ge 1$ and $k,m\in\N$. 

\UUU Given two functions $f, g \colon \R^{n} \to \R$, we write $f \approx g$ if there exists $c \geq 1$ such that $c^{-1} \, g(x) \leq f(x) \leq c\, g(x)$ for every $x \in \R^{n}$. \GGG Similarly the symbol $\lesssim$ stands for $\le$ up to a constant.

\EEE For $\gamma > 1$, $z \in \R^{n}$, and $i = 1, \ldots, n$ we define $W_\gamma^i(z) : =|z_i|^{(\gamma-2)/2}z_i$ and $W_\gamma(z) : =(W_\gamma^1(z),\dots,W_\gamma^i(z),\dots,W_\gamma^n(z))$. The following inequalities hold:
\begin{align}\label{eq:vp}
	& |W_\gamma(z)-W_\gamma(z')|^2\le c\sum_{i=1}^n \left(|z_i|^{\gamma-2}z_i-|z'_i|^{\gamma-2}z'_i\right)(z_i-z'_i),\\
    \label{eq:vp2}
	&|W_\gamma(z)-W_\gamma(z')|\approx \sum_{i=1}^n(|z_i|+|z'_i|)^{(\gamma-2)/2}|z_i-z'_i|\,,
\end{align}
\UUU for a constant~$c>0$ only depending on~$n$ and on~$\gamma$. In particular,~\eqref{eq:vp2} holds componentwise (cf.~\cite{giusti}): \EEE
\begin{equation}\label{eq:vp3}
	|W^i_\gamma(z)-W^i_\gamma(z')|\approx (|z_i|+|z'_i|)^{(\gamma-2)/2}|z_i-z'_i|.
\end{equation}
\UUU We also introduce the function $V_{\gamma} \colon \R^{n} \to \R^{n}$ as $V_\gamma(z)=|z|^{(\gamma-2)/2}z$. Then, $V_{\gamma}$ satisfies the relations
\begin{align}
\label{eq:vp4}
	& |V_\gamma(z)-V_\gamma(z')|^2\le c\langle |z|^{\gamma-2}z-|z'|^{\gamma-2}z',z-z'\rangle\,,\\
    \label{eq:vp5}
	& |V_\gamma(z)-V_\gamma(z')|\approx(|z|+|z'|)^{(\gamma-2)/2}|z-z'|\,,
\end{align}
for every $z, z' \in \R^{n}$, for a positive constant~$c$ only depending on~$n$ and on~$\gamma$. \EEE

\GGG  
In what follows $c$ or $C$ will be often general positive constants, possibly varying from line to line, but depending on only the structural parameters
$n,p,q$, on the function $a$ in terms of $\alpha, [a]_{0,\alpha}, \|a\|_{L^\infty(\Omega)},\|\nabla a\|_{W^{1,\frac{n}{1-\alpha}}(\Omega)}$ and  on the energy controlled by the coercivity of the functional, that is $\|\nabla u\|_{L^p(\Omega)}$. The dependences on some of $\alpha, [a]_{0,\alpha}, \|a\|_{L^\infty(\Omega)}, \|\nabla a\|_{W^{1,\frac{n}{1-\alpha}}(\Omega)}$ or all of these, and on $\|\nabla u\|_{L^p(\Omega)}$ will be simply denoted with $a,$ and $u$, respectively. 

\subsection{Auxiliary lemmas}
\GGG The following technical result is classical in the Regularity Theory. We state it here for the reader's convenience (see \cite[Lemma 6.1]{giusti}). \EEE
\begin{lem}\label{lem:est}
	Let $0<r\le R<1$ and let $Z\colon y\in [r,R]\mapsto Z(y)\in[0,+\infty)$ be a bounded function. Assume that there exist $\cA,\cB,\cC\ge 0$, $\alpha_0\ge\beta_0>0$ and $\theta\in[0,1)$ such that
	\begin{equation*}
		Z(s)\le\frac{\cA}{(t-s)^{\alpha_0}}+\frac{\cB}{(t-s)^{\beta_0}}+ \cC + \theta Z(t)
	\end{equation*}
	holds for all $s,t$ such that $r\le s<t\le R$. Then, for any $\lambda\in\left(\theta^{\frac{1}{\alpha_0}},1\right)$ it holds
	\begin{equation*}
		Z(r)\le\frac{1}{(1-\lambda)^{\alpha_0}}\frac{\lambda^{\alpha_0}}{\lambda^{\alpha_0}-\theta}\left(\frac{\cA}{(R-r)^{\alpha_0}}+\frac{\cB}{(R-r)^{\beta_0}}+\cC\right).
	\end{equation*}
\end{lem}

\subsection{Approximation}\label{s:approximation}
We will use an approximation scheme. 
Let $U \in W^{1, p} (\Omega)$ be a local minimizer of~$\cF_{0}$. 
\GGG We fix a ball 
\[
B_R \Subset \Omega\quad \mbox{ such that }\quad 2B_R\Subset\Omega \mbox{ as well}.
\] 
Here by $\lambda\,B_R$ we denote the ball concentric with $B_R$, scaled by a factor $\lambda>0$. 
We set
\[
\varepsilon_0=\min\left\{1, \frac{R}{2}\right\}>0.
\]
For every $0<\varepsilon\le \varepsilon_0$, we denote by $\varrho_\varepsilon$ a standard mollifier, supported in a ball of radius $\varepsilon$ centered at the origin. For every \(x\in \overline{B}_R\), we then define 
\[
U^\varepsilon(x)=U\ast \varrho_\varepsilon(x).
\] 
We notice that
\begin{equation}\label{eq:C1}
\GGG |\nabla U^\varepsilon|(x)\le \|\nabla U\|_{L^p(\Omega)}C\varepsilon^{-\frac{n}{p}}=C_1\varepsilon^{-\frac{n}{p}}.
\end{equation}
\GGG

For $ v \in W^{1, p} (\Omega)$, and $\omega \subseteq \Omega$ we further define the approximate functional
\begin{equation}\label{eq:appr}
	\cF_\varepsilon[v,\omega]: =\cF_{0}[v,\omega]+\frac{\varepsilon}{1+\|\nabla U^\varepsilon\|^q_{L^q(B_R)}}\int_\omega  |\nabla v(x) |^{q}  \,dx.
\end{equation}

In the sequel we shall use the following lemma based on some arguments due to Zhikov \cite{zhikov} (see also \cite[Lemma 13]{EspLeoMingio}).

\begin{lem}\label{lavr}
Let $a\colon \Omega\to[0,+\infty)$ be a function satisfying \eqref{hpholder} and let $1<p\le q<+\infty$,  be such that $\frac{q}{p}\le 1+\frac{\alpha}{n}$. Then 
\begin{equation}
\lim_{\varepsilon\to 0}\cF_{0}[U^\varepsilon,B_R]=\cF_{0}[U,B_R].
\end{equation}
\end{lem}

\begin{proof}
Define 
\begin{equation*}
\begin{split}
&F(x,z)=|\xi|^p_p+a(x)|z|^q_q,\\
&a_\varepsilon(x)=\inf\{a(y) : y\in B_{R+\varepsilon}, |x-y|\le\varepsilon\},\\
&F_\varepsilon(x,z)=|\xi|^p_p+a_\varepsilon(x)|z|^q_q.
\end{split}
\end{equation*}
Then, by the H\"older continuity of $a(x)$, $a(x)\ge a_\varepsilon(x)\ge a(x)-\varepsilon^\alpha[a]_{0,\alpha}.$ Moreover, by the definition of $a_\varepsilon$ we easily deduce that
\begin{equation}\label{FepsF}
F_\varepsilon(x,z)\le F(y,z) \ \ \forall x,y\in B_{R+\varepsilon}, |x-y|\le \varepsilon.
\end{equation}
On the other hand, given $\delta\in (0,1)$, if $|z|\le C_1\varepsilon^{-\frac{n}{p}}$, we have
\begin{equation*}
\begin{split}
F_\varepsilon(x,z)&=\delta F(x,z)+\delta(a_\varepsilon(x)-a(x))|z|_q^q+(1-\delta)F_\varepsilon(x,z)\\
&\ge \delta F(x,z)-\delta \varepsilon^\alpha[a]_{0,\alpha}|z|_q^q+(1-\delta)F_\varepsilon(x,z)\\
&\ge \delta F(x,z)-\delta \varepsilon^\alpha[a]_{0,\alpha}C_1^{q-p}\varepsilon^{-\frac{n}{p}(q-p)}|z|_p^p+(1-\delta)|z|_p^p\\
&\ge \delta F(x,z)+(1-\delta-\delta[a]_{0,\alpha}C_1^{q-p})|z|_p^p.\\
\end{split}
\end{equation*}
Let us remark that the last estimate relies on the fact that $\frac{q}{p}\le 1+\frac{\alpha}{n}.$ Therefore, choosing $\delta$ small enough, we infer that 
\begin{equation}\label{FFeps}
F(x,z)\le \frac{1}{\delta}F_\varepsilon (x,z) \ \ \ \hbox{ if }|z|\le C_1\varepsilon^{-\frac{n}{p}}.
\end{equation}
The convexity of the function $z\mapsto F_\varepsilon(\cdot,z)$ allows to apply Jensen's inequality in connection with the fact that $U^\varepsilon$ is defined by a convolution. This gives, together with \eqref{FepsF},
\begin{equation}
\begin{split}
F_\varepsilon(x,\nabla U^\varepsilon(x))&\le\int_{B_{R+\varepsilon}}F_\varepsilon(x,\nabla U(y))\rho_\varepsilon(x-y)\,dy\\
&\le \int_{B_{R+\varepsilon}}F(y,\nabla U(y))\rho_\varepsilon(x-y)\,dy=F(\cdot,\nabla  U)\ast\rho_\varepsilon (x). \\
\end{split}
\end{equation}
Then using \eqref{eq:C1} and \eqref{FFeps}
$$
F(x,\nabla U^\varepsilon(x))\le \frac{1}{\delta}F(\cdot,\nabla  U)\ast\rho_\varepsilon (x).
$$
Recalling that $F(\cdot,\nabla  U)\ast\rho_\varepsilon (x)\to  F(x,\nabla U(x))$ strongly in $L^1(B_R)$, the assertion
easily follows by recalling that $U^\varepsilon\to U$ in $W^{1,p}(B_R)$, using a well-known variant of
Lebesgue's dominated convergence theorem. 
\end{proof}

\begin{prop}\label{exist}
	For every $\varepsilon \in (0, \varepsilon_0]$, there exists a unique solution $u^\varepsilon$ to
	\begin{equation}\label{eq:minepsilon}
		\min\left\{\cF_\varepsilon[v,B_R]: \, v\in U^\varepsilon+W^{1,p}_0(B_R)\right\}.
	\end{equation}
	Moreover, this minimizer satisfies the estimate
	\begin{equation}\label{eq:boundp}
		\int_{B_R}|\nabla u^\varepsilon|^p\,dx\le C
	\end{equation}
	for some $C=C(n,p,q,a,U)$.
\end{prop}
\begin{proof}
	The existence of $u^\varepsilon$ follows from a simple application of the direct method, while uniqueness follows \UUU by strict \EEE convexity. By minimality, we have $\cF_\varepsilon[u^\varepsilon,B_R]\le\cF_\varepsilon[U^\varepsilon,B_R]$, so that
	\begin{align*}
		\int_{B_R} |\nabla u^\varepsilon|^p\,dx&\le \GGG n^{(p-2)/2}\cF_\varepsilon[u^\varepsilon,B_R]\\
		&\le \GGG n^{(p-2)/2} \cF_{0}[U^\varepsilon,B_R]+n^{(p-2)/2}\frac{\varepsilon}{1+\|\nabla U^\varepsilon\|^q_{L^q(B_R)}}\int_{B_R}  |\nabla U^\varepsilon(x) |^{q}  \,dx\\
		&\le \GGG n^{(p-2)/2} \cF_{0}[U^\varepsilon,B_R]+n^{(p-2)/2}{\varepsilon} \le C,
	\end{align*}
where the uniform bound follows by Lemma \ref{lavr}.
\end{proof}

\begin{prop}
	\UUU For $\varepsilon \in (0, \varepsilon_0]$, let $u^{\varepsilon}$ be as in Proposition \ref{exist}. Then, \GGG $u^\varepsilon\in W^{1,\infty}_{loc}(B)\cap W^{2,2}_{loc}(B)$. \EEE
\end{prop}
\begin{proof}
This follows by standard regularity theory (see for instance \cite{kuusimingio}, \cite{apdn}), since  the functionals appearing in \eqref{eq:appr} have 
standard polynomial growth of order $q$.

\end{proof}
\begin{prop}\label{prop:approx} 
\GGG With the same notation as above we have
\begin{equation}\label{strongconv}
\lim_{\varepsilon\to 0}\|u^\varepsilon-U\|_{W^{1,p}(B_R)}=0.
\end{equation}
\end{prop}
\EEE
\begin{proof}
	By using the uniform estimate~\eqref{eq:boundp} and the definition of~$U^\varepsilon$ we have that there exists a sequence $\left\{\varepsilon_k\right\}_{k\in\N}$, $0<\varepsilon_k\le\varepsilon_0$, such that $\varepsilon_k\to 0$ and $u^{\varepsilon_k}-U^{\varepsilon_k}$ converges to some $\tilde{u}\in W_0^{1,p}(B_R)$ weakly in~$W_0^{1,p}(B_R)$ and almost everywhere. By recalling that~$U^{\varepsilon_k}$ has been constructed by convolution, we also have that it converges strongly
in $W^{1,p}(B_R)$ and almost everywhere to~$U.$ This permits to conclude that~$u^{\varepsilon_k}$ converges weakly
and almost everywhere to $u:=\tilde{u}+U.$ Since~$U^\varepsilon$ is admissible for problem \eqref{eq:minepsilon}, we have that
	\begin{equation}\label{scifunct}
		\liminf\limits_{k\to\infty} \UUU \cF_{\varepsilon_k} \EEE [U^{\varepsilon_k},B_R]\ge\liminf\limits_{k\to\infty} \UUU \cF_{\varepsilon_k} \EEE [u^{\varepsilon_k},B_R]\ge\liminf\limits_{k\to\infty}\UUU \cF_{0} \EEE [u^{\varepsilon_k},B_R]\ge\UUU \cF_{0} \EEE [{u},B_R]
\end{equation}
by lower semicontinuity of $\cF_{0}$. Next, we notice that  by Lemma \ref{lavr} 
\begin{equation}\label{convfunct}
\begin{split}
	\lim\limits_{k\to\infty}\left|\UUU \cF_{\varepsilon_k} \EEE [U^{\varepsilon_k},B_R]-\UUU \cF_{0} \EEE [U,B_R]\right|&\le\lim\limits_{k\to\infty}\left|\UUU \cF_{0} \EEE [U^{\varepsilon_k},B_R]-\UUU \cF_{0} \EEE [U,B_R]\right|\\
		&+\lim\limits_{k\to\infty}\frac{\varepsilon_k}{1+\|\nabla U^{\varepsilon_k}\|^q_{L^q(B_R)}}\int_{B_R}\left|\nabla U^{\varepsilon_k}\right|^q\,dx=0
	\end{split}
	\end{equation}
	and thus
	\begin{equation*}
		\UUU \cF_{0} \EEE [U,B_R]=\lim\limits_{k\to\infty} \UUU \cF_{\varepsilon_k} \EEE [U^{\varepsilon_k},B_R]\ge\UUU \cF_{0} \EEE [{u},B_R]\,.
	\end{equation*}
	
	\GGG By the strict convexity of the functional, the minimizer must be unique and thus we get ${u}=U,$ as desired. In order to prove \eqref{strongconv} we can adapt the argument of~\cite[Lemma 2.6]{BouBraLeo}. By~\eqref{scifunct}, and taking into account \eqref{convfunct} and the fact that $u=U$, we get
\begin{equation}
\label{norms}
\lim_{k\to \infty} \sum_{i=1}^n\,\int_{B_R} \left(\left|u^{\varepsilon_k}_{x_i}\right|^{p}+a(x)\left|u^{\varepsilon_k}_{x_i}\right|^{q}\right)dx=\sum_{i=1}^n \int_{B_R}\left( \left|U_{x_i}\right|^{p}+
a(x)\left|U_{x_i}\right|^{q}\right)dx.
\end{equation}
For every \(i=1,\dots,n\), we rely on the lower semicontinuity of the $L^{p}$ norm to get
\[
\liminf_{k\to \infty} \int_{B_R} \left| u^{\varepsilon_k}_{x_i}\right|^{p}dx \geq \int_{B_R} \left|U_{x_i}\right|^{p}dx, 
\]
as well as 
\[
\liminf_{k\to \infty} \sum_{i=1}^n\int_{B_R} a(x)\left|u^{\varepsilon_k}_{x_i}\right|^{q}dx \geq \sum_{i=1}^n\int_{B_R}a(x) \left|U_{x_i}\right|^{q}dx.
\]
In connection with \eqref{norms}, this implies that 
\[
\lim_{k\to \infty} \int_{B_R} \left|u^{\varepsilon_k}_{x_i}\right|^{p}dx = \int_{B_R} \left|U_{x_i}\right|^{p}dx.
\]
The convergence of the norms, combined with the weak convergence, permits to infer that~$u^{\varepsilon_k}_{x_i}$ converges to~$U_{x_i}$ in \(L^{p}(B_R)\) for every \(i=1, \dots, n\). Finally, we observe that we can repeat this argument with any subsequence of the original family $\{u^{\varepsilon}\}_{0<\varepsilon\le\varepsilon_0}$. Thus the above limit holds true for the whole family $\{u^{\varepsilon}\}_{0<\varepsilon\le \varepsilon_0}$ instead of $\{u^{\varepsilon_k}\}_{k\in \mathbb{N}}$ and \eqref{strongconv} follows.

\end{proof}

\section{Higher integrability}
\label{s:higher-integrability}

In this section, we establish Theorem \ref{revhol}. We assume throughout the section that $2\le p\le q<+\infty$, the function $a(\cdot)$ satisfies \eqref{hpholder} and $\frac{q}{p}<1+\frac{\alpha}{n}.$


\begin{proof}[Proof of Theorem~\ref{revhol}]
We adapt the proof of Theorem 5.1 in \cite{colombo2015regularity} to our setting.
	
	We claim that if $B_R=B_R(x_0) \UUU \Subset \Omega$ 
	then for every $\beta\in (n(q/p-1),\alpha)$ there exists $c=c(n,p,q,\beta)>0$ such that
	\begin{equation}\label{eq:prerh}
		\left(\fint_{B_{R/2}}|\nabla U|_p^{\frac{np}{n-2\beta}}\,dx\right)^{\frac{n-2\beta}{n}}\le c\fint_{B_R}|\nabla U|_p^p\,dx+\left(\|a\|_{L^\infty(B_R)}^2+R^{2\alpha}[a]_{0,\alpha}^2\right)^{b_1}\left(\fint_{B_R}|\nabla U|_p^p\right)^{b_2}
	\end{equation}
	with
	\begin{equation}\label{eq:bb}
		b_1=\frac{\beta p}{\beta p-n(q-p)}\ge 1,\quad b_2=\frac{\beta(2q-p)-n(q-p)}{\beta p-n(q-p)}\ge 1.
	\end{equation}
	In particular, \eqref{eq:prerh} implies \eqref{eq:rh2} with an appropriate choice of $\beta$.
	
	Under \EEE the transformations
	\begin{equation}\label{e:rescaling}
		\tilde{U}(x)=\frac{U(x_0+Rx)-(U)_{B_R}}{R},\quad \tilde{a}(x)=a(x_0+Rx)
	\end{equation}
	with $(U)_{B_R}$ the integral average of $U$ on $B_R$, 
	we have that $\tilde{U}\in W^{1,p}(\UUU B \EEE )$ is a minimizer of
	\begin{equation*}
		\UUU \widetilde{\cF}_{0} \EEE[\tilde{v}]=\int_{\UUU B \EEE} |\nabla\tilde{v}|_p^p  +  \tilde{a}(x)|\nabla\tilde{v}|_q^q\,dx,
	\end{equation*}
	\UUU where we have set $B = B_{1}$ the ball of center $0$ and radius $1$. We notice that \EEE inequality \eqref{eq:prerh} becomes
	\begin{equation}\label{eq:revholsc}
\left\||\nabla \tilde{U}|_p\right\|_{L^{np/(n-2\beta)}(B_{1/2})}^p\le c\left\||\nabla \tilde{U}|_p\right\|_{L^p( B )}^p+c\left(\|\tilde{a}\|_{L^\infty( B  )}^2+[\tilde{a}]_{0,\alpha}^2\right)^{b_1}\left\||\nabla \tilde{U}|_p\right\|_{L^p( B )}^{pb_2}.
	\end{equation}
	\EEE
	
Next steps aim at proving inequality \eqref{eq:revholsc}. 
	
	\noindent{\bf Step 1.} 
        \GGG
        We are in position to apply Lemma \ref{lavr} and to consider a sequence $\left\{u_m\right\}_{m\in\N}\subset W^{1,\infty}(B)$ such that
		\begin{equation*}
			u^m\to \tilde{U}\text{ in }W^{1,p}(B),\quad \UUU \widetilde{\cF}_{0} [u^m] \EEE \to \UUU \widetilde{\cF}_{0}[\tilde{U}] 
		\end{equation*}
		\EEE where
		\begin{equation*}
			\UUU \widetilde{\cF}_{0} \EEE [v] : =\int_B (|\nabla v|_p^p+\tilde{a}|\nabla v|_q^q)\,dx \qquad \UUU \text{for $v \in W^{1, p} (B)$}. \EEE
		\end{equation*}
		\UUU For $m \in \mathbb{N}$, $x \in B$, and $z \in \R^{n}$ \EEE we now denote
		\begin{equation*}
			\begin{cases}
				\sigma_m :=\left(1+m+\|\nabla u^m\|_{L^{2q-p}(B)}^{2(2q-p)}\right)^{-1}\\
				F(x,z) = |z|_p^p+\tilde{a}(x)|z|_q^q\\
				F_m(x,z)= F(x,z) + \EEE \sigma_m|z|^{2q-p}.
			\end{cases}
		\end{equation*}
		The functional
		\begin{equation*}
			\fF_m[w]=\int_B F_m(x,\nabla w)\,dx \qquad \UUU \text{for $w \in W^{1, p} (B)$} \EEE
		\end{equation*}
		behaves like a standard functional of growth $2q-p$, and therefore admits a unique minimizer $v^m\in u^m+W^{1,2q-p}(B)$, which by standard techniques belongs to $W^{1,\infty}_{loc}(B)$ (see for instance~\cite{kuusimingio}).
		
        \noindent{\bf Step 2.} We prove that for $\beta\in(0,\alpha)$ there exists \UUU a constant $c=c(n,p,q,\beta)>0$ \EEE such that \UUU for every $m \in \mathbb{N}$ \EEE
		\begin{equation}
        \label{e:claim-vm}
			\||\nabla v^m|_p\|_{L^{np/(n-2\beta)(B_{1/2})}}^p\le c\||\nabla v^m|_p\|_{L^p(B)}^p+c\left(\|\tilde{a}\|_{L^\infty(B)}+[\tilde{a}]_{0,\alpha}+\sigma_m\right)\||\nabla v^m|_p\|_{L^{2q-p}(B)}^{2q-p}.
		\end{equation}
		Consider the Euler-Lagrange equation for $\fF_m$:
		\begin{equation}\label{eq:Fm_EL}
			\int_B\langle\nabla_z F_m(x,\nabla v_m),\nabla\varphi\rangle\,dx=0\qquad \text{for every }\varphi\in W^{1,2q-p}_0(B).
		\end{equation}
		\UUU Let \EEE $\eta\in C^\infty_0(B)$ be such that $0\le\eta\le 1$, $\eta\equiv 0$ outside $B_{3/4}$, $\eta\equiv 1$ on $B_{2/3}$ and $|\nabla\eta|^2+|D^2\eta|\le M$, for some $M$. For $h\in\R^n$ sufficiently small, e.g., $|h|<10^{-4}$, we test \eqref{eq:Fm_EL} with $\varphi=\tau_{-h}(\eta^2\tau_hv^m)$, \UUU where $\tau_{h} v (x) : = v ( x+h) - v (x)$ for $x \in B$. \EEE We obtain
		\begin{align*}
			0=&\int_B\langle\nabla_z F_m(x,\nabla v^m),\nabla\left(\tau_{-h}(\eta^2\tau_hv^m)\right)\rangle\,dx\\
			=&\int_B\langle\nabla_z F_m(x,\nabla v^m),\tau_{-h}\nabla\left(\eta^2\tau_hv^m\right)\rangle\,dx\\
			=&\int_B\langle\nabla_z F_m(x,\nabla v^m),\nabla\left(\eta^2\tau_hv^m\right) \GGG (x-h) \EEE\rangle\,dx\\
			&-\int_B\langle\nabla_z F_m(x,\nabla v^m),\nabla\left(\eta^2\tau_hv^m\right) \GGG (x) \EEE \rangle\,dx\\
		       =& \int_B\langle\nabla_z F_m(\UUU x + h \EEE ,\nabla v^m \UUU(x+h) \EEE),\nabla\left(\eta^2\tau_hv^m\right)(x)\rangle\,dx\\
		        &-\int_B\langle\nabla_z F_m(\UUU x \EEE ,\nabla v^m\UUU(x)\EEE),\nabla\left(\eta^2\tau_hv^m\right)(x)\rangle\,dx\\
                        =&\int_B\langle\tau_h \big( \nabla_z F_m(\cdot,\nabla v^m) \big) ,\nabla(\eta^2\tau_h v^m)\rangle\,dx
		\end{align*}
		We write $\tau_h \big( \nabla_z F_m(\cdot,\nabla v^m) \big) =A_1+A_2$ where, for every $z\in\mathbb{R}^n$ and $x \in B$,
		\begin{align*}
			\langle A_1 (x) ,z\rangle&=\langle\tau_h \big[ \nabla_z F_m( x + h ,\nabla v^m(\cdot )  ) \big] (x) ,z\rangle\\
			&=p\sum_{i=1}^n\left(|v^m_{x_i}(x+h)|^{p-2}v^m_{x_i} \UUU (x + h) \EEE -|v^m_{x_i}(x)|^{p-2}v^m_{x_i}(x)\right)z_i\\
			&+ \UUU \tilde{a} \EEE (x+h)q\sum_{i=1}^n\left(|v^m_{x_i}(x+h)|^{q-2}v^m_{x_i} \UUU (x + h) \EEE -|v^m_{x_i}(x)|^{q-2}v^m_{x_i}(x)\right)z_i \\
			&+\sigma_m(2q-p) \sum_{i=1}^n\left( |\nabla v^m  (x + h) |^{2q-p-2}v^m_{x_i} (x + h)  - |\nabla v^m (x) |^{2q-p-2}v^m_{x_i} (x ) \right)z_i
		\end{align*}
		and
		\begin{align*}
			\langle A_2 (x) ,z\rangle &=\langle\tau_h \big[\nabla_z F_m(\cdot,\nabla v^m(x)) \big] (x) ,z\rangle\\
			&=(\UUU \tilde{a}(x+h)-\tilde{a} \EEE (x))q\sum_{i=1}^n|v^m_{x_i}(x)|^{q-2}v^m_{x_i}(x)z_i. 
		\end{align*}
		\UUU Thus, \EEE we can write
		\begin{equation}  \label{e:first-sum}
        \begin{split}
			I_0&:=\int_B  \UUU \eta^{2} \EEE \langle A_1,\tau_h\nabla v^m\rangle  \,dx\\
			&=-\int_B\eta^2\langle A_2,\tau_h\nabla v^m\rangle\,dx-2\int_B\eta\langle A_2,\nabla \eta\rangle\tau_hv^m\,dx-2\int_B\eta\langle A_1,\nabla \eta\rangle\tau_hv^m\,dx\\
			&=:I_1+I_2+I_3. 
			\end{split}
		\end{equation}
		We estimate $I_0$ by using \eqref{eq:vp} and \eqref{eq:vp4}:
		\begin{equation} \label{e:I0}
		\begin{split}
			I_0&\gtrsim\int_B(\eta^2|\tau_h W_p(\nabla v^m)|^2+ \UUU \tilde{a}(x + h ) \EEE \eta^2|\tau_hW_q(\nabla v^m)|^2+\sigma_m\eta^2|\tau_hV_{2q-p}(\nabla v^m)|^2)\,dx\\
			&\ge\int_B(\eta^2|\tau_h W_p(\nabla v^m)|^2+\sigma_m\eta^2|\tau_hV_{2q-p}(\nabla v^m)|^2)\,dx.
		\end{split}
		\end{equation}
		We estimate $I_1$ using the H\"older continuity of $\tilde{a}$:
		\begin{align*}
			|I_1|&\lesssim[\tilde{a}]_{0,\alpha}|h|^\alpha\int_B\eta^2\sum_{i=1}^n|v^m_{x_i}(x)|^{q-1}|\tau_hv^m_{x_i}(x)|\,dx\\
			&\le[\tilde{a}]_{0,\alpha}|h|^\alpha\int_B\eta^2\sum_{i=1}^n(|v^m_{x_i}(x+h)|+|v^m_{x_i}(x)|)^{q-1}|\tau_hv^m_{x_i}(x)|\,dx.
		\end{align*}
		\UUU Writing $q-1 = ( q- \frac{p}{2}) + (\frac{p}{2} - 1)$, applying a weighted \EEE Young's inequality, inequality \eqref{eq:vp2}, and the fact that $\eta\equiv 0$ outside of $B_{3/4}$, \UUU for $\delta>0$ we \EEE obtain
		\begin{equation}\label{e:I1}
		\begin{split}
			|I_1|&\lesssim\delta\int_B\eta^2\sum_{i=1}^n(|v^m_{x_i}(x+h)|+|v^m_{x_i}(x)|)^{p-2}|\tau_hv^m_{x_i}(x)|^2\,dx\\
			&+\frac{[\tilde{a}]_{0,\alpha}^2|h|^{2\alpha}}{\delta}\int_{B_{3/4}}\sum_{i=1}^n (|v^m_{x_i}(x+h)|+|v^m_{x_i}(x)|)^{2q-p}\,dx \\
			&\lesssim\delta\int_B\eta^2|\tau_hW_p(\nabla v^m)|^2\,dx+\frac{[\tilde{a}]_{0,\alpha}^2|h|^{2\alpha}}{\delta}\int_B |\nabla v^m|^{2q-p}\,dx.
			\end{split}
		\end{equation}
		Using the fact that $|V_{2q-p}(\nabla v^m)|^2=|\nabla v^m|^{2q-p}$ we have
		\begin{equation*}
			|I_1|\lesssim \delta I_0+\frac{[\tilde{a}]_{0,\alpha}^2|h|^{2\alpha}}{\delta}\int_B \UUU |V_{2q-p}(\nabla v^m)|^{2} \EEE\,dx.
		\end{equation*}
		To estimate $I_2$ we recall that for a function $u\in W^{1,\gamma}(B_R),$ $\gamma\ge 1,$ and $B_\rho\subset B_R$ a concentric ball,  the following basic property of
finite differences holds
		\begin{equation*}
			\int_{B_\rho}|\tau_hu|^\gamma\,dx\le|h|^\gamma\int_{B_R}|\nabla u|^\gamma\,dx,
		\end{equation*}
		for $|h|\le R-\rho.$
		 Hence, by Young and H\"older inequality we estimate $I_2$ as 
		\begin{align*}
			|I_2|&\lesssim [\tilde{a}]_{0,\alpha}|h|^\alpha\int_{B_{3/4}}|\nabla v^m|^{q-1}|\tau_hv^m||\nabla\eta|\,dx\\
			&\lesssim|h|^\alpha\left(\int_{B_{3/4}}|\nabla v^m|^{p-1}|\tau_hv^m|\,dx + [\tilde{a}]^{2}_{0,\alpha}  \int_{B_{3/4}}|\nabla v^m|^{2q-p-1}|\tau_hv^m|\,dx\right)\\
			&\lesssim|h|^\alpha\left(\int_B  |\nabla v^m|^{p} \,dx\right)^{\frac{p-1}{p}}\left(\int_{B_{3/4}}|\tau_hv^m|^p\right)^{\frac{1}{p}}\\
			&+ [\tilde{a}]^{2}_{0,\alpha}  |h|^\alpha\left(\int_B|\nabla v^m|^{2q-p}\right)^{\frac{2q-p-1}{2q-p}}\left(\int_{B_{3/4}}|\tau_hv^m|^{2q-p}\right)^{\frac{1}{2q-p}}\\
			&\lesssim|h|^{2\alpha}\left(\int_B|\nabla v^m|^p+ [\tilde{a}]^{2}_{0,\alpha} |\nabla v^m|^{2q-p}\,dx\right)\\
			&\lesssim|h|^{2\alpha}\left(\int_B|\nabla v^m|_p^p+  [\tilde{a}]^{2}_{0,\alpha}  |\nabla v^m|^{2q-p}\,dx\right)\\
		\end{align*}
		where we used the fact that $|h|^{1+\alpha}\le|h|^{2\alpha}$.
        
		We are now left with $I_3$. \UUU Since $p\ge 2$, we \EEE have
		\begin{equation}
			|I_3|\lesssim\|\nabla \eta\|_{L^\infty(B)}(I_{3,p}+\|\tilde{a}\|_{L^\infty(B)}I_{3,q}+\sigma_mJ_{3,2q-p})
		\end{equation}
		where
		\begin{equation*}
			I_{3,\gamma}:=\int_B\eta|\tau_hv^m|\sum_{i=1}^n(|v^m_{x_i}(x+h)|+|v^m_{x_i}(x)|)^{\gamma-2}|\tau_h v^m_{x_i}|\,dx
		\end{equation*}
		for $\gamma=p,q$, and
		\begin{equation*}
			J_{3,\gamma}:=\int_B\eta(|\nabla v^m(x+h)|+|\nabla v^m(x)|)^{\gamma-2}|\tau_h\nabla v^m||\tau_hv^m|\,dx
		\end{equation*}
		for $\gamma=2q-p$.
		Let us estimate these terms. Using Young's inequality \UUU as in~\eqref{e:I1} and recalling~\eqref{eq:vp2} \EEE we have
		\begin{equation}\label{e:I3q1}
		\begin{split}
			\|\tilde{a}\|_{L^\infty(B)}I_{3,q}\lesssim& \ \delta\int_{B_{3/4}}\eta^2\sum_{i=1}^n(|v^m_{x_i}(x+h)|+|v^m_{x_i}(x)|)^{p-2}|\tau_hv^m_{x_i}|^2\,dx\\
			&+\frac{\UUU \|\tilde{a}\|^{2}_{L^\infty(B)}\EEE }{\delta}\int_{B_{3/4}}|\tau_hv^m|^2\sum_{i=1}^n(|v^m_{x_i}(x+h)|+|v^m_{x_i}(x)|)^{2q-p-2}\,dx \\
			\lesssim& \ \delta\int_B\eta^2|\tau_hW_p(\nabla v_m)|^2\,dx \\
			&+\frac{\|\tilde{a}\|_{L^\infty(B)}^2}{\delta}\int_{B_{3/4}}|\tau_hv^m|^2\sum_{i=1}^n(|v^m_{x_i}(x+h)|+|v^m_{x_i}(x)|)^{2q-p-2}\,dx. 
		\end{split}
		\end{equation}
		\UUU By H\"older inequality we further estimate \EEE
		\begin{equation}\label{e:I3q2}
		\begin{split}
			\int_{B_{3/4}}&|\tau_hv^m|^2\sum_{i=1}^n(|v^m_{x_i}(x+h)|+|v^m_{x_i}(x)|)^{2q-p-2}\,dx\\
			&\lesssim\int_{B_{3/4}}|\tau_hv^m|^2(|\nabla v^m(x+h)|+|\nabla v^m(x)|)^{2q-p-2}\,dx\\
			&\lesssim\|\nabla v^m\|_{L^{2q-p}(B)}^{2q-p-2}\|\tau_hv^m\|_{L^{2q-p}(B)}^2 \le |h|^2\|\nabla v^m\|_{L^{2q-p}(B)}^{2q-p}.
		\end{split}
		\end{equation}
		\UUU We deduce from~\eqref{e:I3q1} and~\eqref{e:I3q2} that \EEE
		\begin{equation}
			\|\tilde{a}\|_{L^\infty(B)}I_{3,q}\lesssim\delta I_0+\frac{|h|^2\|\tilde{a}\|_{L^\infty(B)}^2}{\delta}\int_B|\nabla v^m|^{2q-p}\,dx.
		\end{equation}
		\UUU Arguing in a similar way we conclude that \EEE 
		\begin{equation} \label{e:sum-of-I}
			I_{3,p} +\sigma_mJ_{3,2q-p} 
             \lesssim\delta I_0+\frac{|h|^2}{\delta}\int_B|\nabla v^m|_p^p\,dx+\frac{|h|^2\sigma_m}{\delta}\int_B|\nabla v^m|^{2q-p}\,dx.
		\end{equation}
		Combining the inequalities~\eqref{e:first-sum}--\eqref{e:sum-of-I}, using \eqref{eq:pequiv} and the relation $|h|^2\le |h|^{2\alpha}$ we get
		\begin{equation*}
			I_0\lesssim \delta I_0+\frac{|h|^{2\alpha}}{\delta}\int_B|\nabla v^m|_p^p\,dx+\frac{|h|^{2\alpha}}{\delta}T_m\int_B|\nabla v^m|_{2q-p}^{2q-p}\,dx,
		\end{equation*}
		where
		\begin{equation*}
			T_m:=\|\tilde{a}\|_{L^\infty(B)}^2+[\tilde{a}]_{0,\alpha}^2+\sigma_m\,.
		\end{equation*}
		\UUU Therefore, choosing $\delta>0$ small enough,~\eqref{e:I0} yields \EEE
		\begin{equation}\label{e:32}
			\int_B\eta^2|\tau_hW_p(\nabla v^m)|^2\,dx\le c|h|^{2\alpha}\left(\int_B|\nabla v^m|_p^p\,dx+T_m\int_B|\nabla v^m|_{2q-p}^{2q-p}\,dx\right).
		\end{equation}
		\UUU Inequality~\eqref{e:32} \EEE implies that $W_p(\nabla v_m)$ belongs to the Nikol'ski\u{\i} space $\mathcal N^{\alpha,2}(B_{2/3})$ and the validity of the embedding $\mathcal N^{\alpha,2}\hookrightarrow W^{\beta,2}\cap L^{2n/(n-2\beta)}$ for any $\beta\in (0,\alpha)$ (see \cite{nik}) implies
		\begin{equation}\label{e:Wp-estimate-1}
			\|W_p(\nabla v^m)\|_{L^{\frac{2n}{n-2\beta}}(B_{1/2}) }^2\le c\left(\left\||\nabla v^m|_p\right\|_{L^p(B)}^p+\|W_p(\nabla v^m)\|_{L^2(B)}^2+T_m\left\||\nabla v^m|_p\right\|_{L^{2q-p}(B)}^{2q-p}\right)
		\end{equation}
		\UUU for some constant~$c=c(n,p,q,\beta)>0.$ Using the equality $|W_p(z)|^2=|z|_p^p$ and~\eqref{eq:pequiv},~\eqref{e:Wp-estimate-1} becomes \EEE
		\begin{equation}\label{eq:cac}
			\left\||\nabla v^m|_p\right\|_{L^{\frac{np}{n-2\beta}}(B_{1/2})}^p\le c\left(\left\||\nabla v^m|_p\right\|_{L^p(B)}^p+T_m\left\||\nabla v^m|_p\right\|_{L^{2q-p}(B)}^{2q-p}\right),
		\end{equation}
        which is precisely~\eqref{e:claim-vm}.
		
        \noindent{\bf  Step 3.} We claim that, for a fixed $\beta\in (0,\alpha)$, there exists a constant $c_\ast=c_\ast(n,p,q,\beta)>0$ such that for every $1/2\le t<s\le 1$ the inequality
		\begin{equation}\label{eq:cac1}
			\left\||\nabla v^m|_p\right\|_{L^{\frac{np}{n-2\beta}}(B_t)}^p\le c_\ast\left(\frac{\left\||\nabla v^m|_p\right\|_{L^p(B_s)}^p}{(s-t)^{2\beta}}+\frac{T_m\left\||\nabla v^m|_p\right\|_{L^{2q-p}(B_s)}^{2q-p}}{(s-t)^{2\beta}}\right)
		\end{equation}
		holds. To this aim, let us fix $s,t$ and let $r:=(s-t)/8$; we consider a covering of~$B_t$ with balls $\left\{B_{r/2}(y_i)\right\}_{i\in I}$ with $y_i\in B_t$ for every $i$, where $\UUU \# I \EEE \le c(n)r^{-n}$ for some dimensional constant $c(n)>0$. We can assume that this family, as well as the family $\left\{B_r(y_i)\right\}_{i\in I}$, satisfy the finite intersection property: each ball of the family intersects at most $c'(n)$ other balls belonging to the family. Using \UUU again the rescaling argument~\eqref{e:rescaling}, \EEE we get
		\begin{equation*}
			\left\||\nabla v^m|_p\right\|_{L^{\frac{np}{n-2\beta}}(B_{r/2}(y_i))}^p\le \UUU {c}\EEE \left(\frac{\left\||\nabla v^m|_p\right\|_{L^p(B_r(y_i))}^p}{r^{2\beta}}+\frac{T_m\left\||\nabla v^m|_p\right\|_{L^{2q-p}(B_r(y_i))}^{2q-p}}{r^{2\beta}}\right).
		\end{equation*}
        for a positive constant $c=c(n,p,q,\beta)$. Summing over every ball of the family, using the finite intersection property and \eqref{eq:pequiv} we get \eqref{eq:cac1}.
		
        \noindent{\bf Step 4.} We now want to prove that, for every $\beta\in\left(n\left(\frac{q}{p}-1\right),\alpha\right)$, \UUU there exists a constant $c=c(n,p,q,\beta)>0$ \EEE such that
		\begin{equation}\label{eq:revholapp}
			\left\||\nabla v^m|_p\right\|_{L^{\frac{np}{n-2\beta}}(B_{1/2})}^p\le c\left\||\nabla v^m|_p\right\|_{L^p(B)}^p+ c \, T_m^{b_1}\left\||\nabla v^m|_p\right\|_{L^p(B)}^{pb_2}
		\end{equation}
		with $b_1,b_2$ as in \eqref{eq:bb}.
		
		Assume $q>p$, since if $q=p$ the previous inequality becomes \eqref{eq:cac}. \UUU Let us fix $\theta \in (0,1)$ \EEE such that
        \begin{equation*}
			\frac{1}{2q-p}=\frac{1-\theta}{p}+\frac{\theta(n-2\beta)}{np} \ \Longrightarrow  \ \theta=\frac{n(q-p)}{\beta(2q-p)}\,.
		\end{equation*}
        By interpolation inequality, we obtain 
		\begin{equation*}
			\left\||\nabla v^m|_p\right\|_{L^{2q-p}(B_s)}\le\left\||\nabla v^m|_p\right\|_{L^p(B_s)}^{1-\theta}\left\||\nabla v^m|_p\right\|_{L^{np/(n-2\beta)}(B_s)}^\theta
		\end{equation*}
		Using Young's inequality with exponents $\frac{p}{\theta(2q-p)}$ and $\frac{p}{p-(2q-p)\theta}$ we infer that, \UUU for every $\delta \in (0, 1),$ \EEE
		\begin{align*}
			\frac{ \UUU T_m \EEE\left\||\nabla v^m|_p\right\|_{L^{2q-p}(B_s)}^{2q-p}}{(s-t)^{2\beta}}&\le\frac{T_m}{(s-t)^{2\beta}}\left\||\nabla v^m|_p\right\|_{L^p(B_s)}^{(1-\theta)(2q-p)}\left\||\nabla v^m|_p\right\|_{L^{\frac{np}{n-2\beta}}(B_s)}^{\theta(2q-p)}\\
			&\le \UUU \delta \EEE \left\||\nabla v^m|_p\right\|_{L^{\frac{np}{n-2\beta}}(B_s)}^{p}+ \UUU c_\delta \EEE\left(\frac{T_m}{(s-t)^{2\beta}}\left\||\nabla v^m|_p\right\|^{p\frac{b_2}{b_1}}_{\UUU L^{p} (B_s) \EEE } \right)^{\UUU b_1 \EEE}.
		\end{align*}
         for some positive constant~$c_{\delta}$ depending on~$\delta,n,p,q,\beta$. We substitute the above \EEE inequality in~\eqref{eq:cac1} with the \UUU choice $\delta=\frac{1}{2c_\ast}$ and obtain
		\begin{equation*}
			\left\||\nabla v^m|_p\right\|_{L^{\frac{np}{n-2\beta}}(B_t)}^p\le\frac{1}{2}\left\||\nabla v^m|_p\right\|_{L^{\frac{np}{n-2\beta}}(B_s)}^p+ \UUU {c} \EEE \left(\frac{\left\||\nabla v^m|_p\right\|_{L^{p}(B)}^p}{(s-t)^{2\beta}}+\frac{T_m^{b_1}\left\||\nabla v^m|_p\right\|_{L^{p}(B)}^{pb_2}}{(s-t)^{2\beta b_1}}\right)
		\end{equation*}
        \UUU for a positive constant ${c} = c(n,p,q,\beta)$. \EEE Using Lemma \ref{lem:est} with $r=1/2$ and $R=1-\UUU \rho \EEE$ for some $\UUU \rho \EEE \in (0,1/2)$ and letting $\UUU \rho \EEE \to 0$ we conclude the proof of \eqref{eq:revholapp}.

        \noindent{\bf Step 5.} We argue as in the proofs of Proposition \ref{exist} and Proposition \ref{prop:approx}. By the minimality of $v^m$ we have
		\begin{align*}
			\int_B|\nabla v^m|_p^p\,dx&\le\int_B F_m(x,\nabla v^m)\,dx \le\int_B F_m(x,\nabla u^m)\,dx\\
            &
            \le \UUU \widetilde{\cF}_{0} \EEE [u^m]+\sigma_m\|\nabla u^m\|_{L^{2q-p}(B)}^{2q-p}\,.
		\end{align*}
		By the definition of $u^m$ and $\sigma_m$, the sequence $v^m$ is bounded in $W^{1,p}(B)$ and, up to a subsequence, $\nabla v^m$ converges weakly to $\nabla w$ for some $w\in W^{1,p}(B)$. Since, \UUU by construction, \EEE $\sigma_m\|\nabla u_m\|_{L^{2q-p}(B)}^{2q-p}\to 0$, we have
		\begin{equation*}
			\UUU \widetilde{\cF}_{0} [w] \le \liminf_{m\to\infty}\fF_{m} [v^m]\le  \limsup_{m\to\infty}\fF_{m} [v^m] \le \widetilde{\cF}_{0}[\tilde{U}]\,. \EEE
		\end{equation*}
		\UUU By the minimality of $\tilde{U}$ and the strict convexity of $\widetilde{\cF}_{0}$ \EEE we obtain $w=\tilde{U}$. In particular
		\begin{align*}
			\UUU \widetilde{\cF}_{0} \EEE [\tilde{U}] =\lim_{m\to\infty}\, \fF_{m} [v^{m}] 
		\end{align*}
		and this implies that 
    	\begin{equation*}
			\lim_{m\to\infty} \int_B|\nabla v^m|_p^p\,dx = \int_B|\nabla\tilde{U}|_p^p\,dx\,.
		\end{equation*}
		Thus,  we can pass to the limit in \eqref{eq:revholapp} to obtain \eqref{eq:prerh}, concluding the proof of \eqref{eq:rh2}.

\end{proof}


\section{Local energy estimates}
\label{s:estimates}

\subsection{Caccioppoli-type inequalities}
\label{s:caccioppoli}

\UUU For $\varepsilon \in (0, \varepsilon_0]$ we consider from now on the minimizer~$u^{\varepsilon}$ of~\eqref{eq:minepsilon}. The corresponding \EEE Euler-Lagrange equation reads as
\begin{equation}\label{eq:ELA}
	\sum_{i=1}^n\int \left(p|u^{\varepsilon}_{x_i}|^{p-2}+q\sigma_\varepsilon |\nabla u^{\varepsilon}|^{q-2}+qa |u^{\varepsilon}_{x_i}|^{q-2}\right)u^{\varepsilon}_{x_i}\varphi_{x_i}\,dx=0,\ \ \ \ \forall\varphi\in W^{1,q}_0(B_R),
\end{equation}
where we denoted $\frac{\varepsilon}{1+\|\nabla U^\varepsilon\|_{L^q(B_R)}^q}$ by $\sigma_\varepsilon.$
For every $j=1,\cdots n$, we can use test functions $\varphi_{x_j}$, with $\varphi\in C^2$ compactly
supported in $B_R,$  and integrate by parts to obtain \GGG
\begin{equation}\label{eq:ELAD}
\begin{split}
&\sum_{i=1}^n\int  \left[p(p-1)|u^{\varepsilon}_{x_i}|^{p-2}+q\sigma_\varepsilon |\nabla u^{\varepsilon}|^{q-2}+q(q-1)a|u^{\varepsilon}_{x_i}|^{q-2}\right]u^{\varepsilon}_{x_ix_j}\varphi_{x_i}\,dx  \\
&+q(q-2)\sigma_\varepsilon \int|\nabla u^\varepsilon|^{q-2}\langle \frac{\nabla u^\varepsilon\otimes\nabla u^\varepsilon}{|\nabla u^\varepsilon|^2}\nabla u^\varepsilon_{x_j},\nabla\varphi\rangle\,dx+ \UUU q \EEE \sum_{i=1}^n\int \UUU a_{x_j} \EEE |u^{\varepsilon}_{x_i}|^{q-2}u^{\varepsilon}_{x_i}\varphi_{x_i}\,dx=0. 
\end{split}
\end{equation}
\UUU From now on, we denote by $B$ the (generic) open ball $B_R$ in~$\R^{n}$. We further use the notation
\begin{equation}\label{e:Miepsilon}
	\cM^{\varepsilon}_i(\nabla u^{\varepsilon}):= p(p-1)|u^{\varepsilon}_{x_i}|^{p-2}+\GGG q\sigma_\varepsilon|\nabla u^{\varepsilon}|^{q-2}\UUU+q(q-1)a |u^{\varepsilon}_{x_i}|^{q-2}
\end{equation}
for $i = 1, \ldots, n$.\EEE

\begin{prop}\label{prop:cac}
	\UUU For $\varepsilon \in (0, \varepsilon_0)$, let $u^\varepsilon$ be the minimizer of \eqref{eq:minepsilon} and assume~$a(\cdot)$ satisfies~\eqref{hpsobolev}. Then, there exists a positive constant $C=C(p,q,n)>0$ such that for every $\Lambda \in C^1 (\R)$ monotone non-decreasing and every Lipschitz function $\eta$ compactly supported in $B$ it holds \EEE
	\begin{equation}\label{eq:caca}
	\begin{split}
			\sum_{i=1}^n \int\cM^{\varepsilon}_i(\nabla u^{\varepsilon}) \Lambda'(u^\varepsilon_{x_j})\left(u^\varepsilon_{x_ix_j}\right)^2\eta^2\,dx \le& \ C\sum_{i=1}^n\int \UUU \cM^{\varepsilon}_{i} (\nabla u^{\varepsilon}) \EEE 
            \frac{\Lambda^2(u^\varepsilon_{x_j})}{\Lambda'(u^\varepsilon_{x_j})}\GGG |\nabla\eta|^2 \EEE\,dx \\
			& +C\sum_{i=1}^n \int a_{x_j}^2|u^\varepsilon_{x_i}|^{2q-p}\Lambda'(u^\varepsilon_{x_j})\eta^2\,dx. 
	\end{split}
	\end{equation}
\end{prop}

\begin{proof}
	\UUU Let us drop the dependence on~$\varepsilon$ in $u^\varepsilon$ and $\cM^\varepsilon_i.$ We test \EEE equation~\eqref{eq:ELAD} with test function $\varphi=\Lambda(u_{x_j})\eta^2$ with $\eta$ \UUU as in the statement of the proposition: \EEE
	\begin{align}\label{e:1000}
		&\sum_{i=1}^n\int \UUU \cM_{i} (\nabla u) \EEE \Lambda'(u_{x_j})u_{x_ix_j}^2\eta^2\,dx \GGG+q(q-2)\sigma_\varepsilon\int|\nabla u|^{q-2}\frac{\langle \nabla u,\nabla u_{x_j}\rangle^2}{|\nabla u|^2} \, \Lambda' (u_{x_{j}}) \, \eta^2\,dx\\		
		&\EEE =-2\sum_{i=1}^n\int \UUU \cM_{i} (\nabla u) \EEE \Lambda (u_{x_j}) u_{x_ix_j}\eta\eta_{x_i}\, dx\GGG -2q(q-2)\sigma_\varepsilon\int|\nabla u|^{q-2}\langle \frac{\nabla u\otimes\nabla u}{|\nabla u|^2}\nabla u_{x_j},\nabla\eta\rangle\eta\Lambda(u_{x_j})\,dx \nonumber\\
		&  -q\sum_{i=1}^n \int a_{x_j}|u_{x_i}|^{q-2}u_{x_i}\Lambda'(u_{x_j})u_{x_ix_j}\eta^2\,dx  -2q\sum_{i=1}^n \int a_{x_j}|u_{x_i}|^{q-2}u_{x_i}\Lambda (u_{x_j})\eta \eta_{x_i}\,dx. \nonumber 
	\end{align}
	\GGG We first observe that the second term on left hand side is nonnegative. \EEE
	Using a weighted Young's inequality, \UUU we estimate the third term on the right-hand side of~\eqref{e:1000}: \EEE
	\begin{align*}
		-\sum_{i=1}^n & \int a_{x_j}|u_{x_i}|^{q-2}u_{x_i}u_{x_ix_j}\Lambda'(u_{x_j})\eta^2\,dx \le\sum_{i=1}^n\int|a_{x_j}||u_{x_i}|^{q-1}|u_{x_ix_j}|\Lambda'(u_{x_j})\eta^2\,dx\\
		&\le\delta\sum_{i=1}^n\int p(p-1)|u_{x_i}|^{p-2}u_{x_ix_j}^2\Lambda'(u_{x_j})\eta^2\,dx + C_\delta \sum_{i=1}^n\int a_{x_j}^2|u_{x_i}|^{2q-p}\Lambda'(u_{x_j})\eta^2\,dx\\
		&\le\delta\sum_{i=1}^n\int \UUU \cM_{i} (\nabla u) \EEE u_{x_ix_j}^2\Lambda'(u_{x_j})\eta^2\,dx+C_\delta \sum_{i=1}^n\int a_{x_j}^2|u_{x_i}|^{2q-p}\Lambda'(u_{x_j})\eta^2\,dx.
	\end{align*}
	\GGG We treat the first two terms on the right hand side: \EEE by Young's inequality, for every $\delta >0$, we have 
	\begin{align*}
		-2\sum_{i=1}^n \int &  \UUU \cM_{i} (\nabla u) \EEE \Lambda(u_{x_j})u_{x_ix_j}\eta\eta_{x_i}\, dx \GGG -2q(q-2)\sigma_\varepsilon\int|\nabla u|^{q-2}\langle \frac{\nabla u\otimes\nabla u}{|\nabla u|^2}\nabla u_{x_j},\nabla\eta\rangle\eta\Lambda(u_{x_j})\,dx\\
		&\GGG \le C\sum_{i=1}^n\int\cM_i(\nabla u)|\Lambda(u_{x_j})||u_{x_ix_j}||\eta||\nabla \eta|\,dx\\
		\le\delta&\sum_{i=1}^n\int \cM_{i} (\nabla u) \Lambda'(u_{x_j})u_{x_ix_j}^2\eta^2\,dx + C_\delta\sum_{i=1}^n\int \cM_{i} (\nabla u)\frac{\Lambda^2(u_{x_j})}{\Lambda'(u_{x_j})}|\nabla\eta|^2\,dx. 
	\end{align*}
	Similarly we get
	\begin{align*}
		-2q&\sum_{i=1}^n\int a_{x_j}|u_{x_i}|^{q-2}u_{x_i}\Lambda (u_{x_j})\eta\eta_{x_i}\,dx\\
		& \lesssim\sum_{i=1}^n\int  a_{x_j}^2|u_{x_i}|^{2q-p}\Lambda'(u_{x_j})\eta^2\,dx+\sum_{i=1}^n\int p(p-1)|u_{x_i}|^{p-2}\frac{\Lambda^2(u_{x_j})}{\Lambda'(u_{x_j})}\eta_{x_i}^2\,dx\\
		& \lesssim\sum_{i=1}^n \int a_{x_j}^2|u_{x_i}|^{2q-p}\Lambda'(u_{x_j})\eta^2\,dx+\sum_{i=1}^n\int \UUU \cM_{i} (\nabla u) \EEE \frac{\Lambda^2(u_{x_j})}{\Lambda'(u_{x_j})} \UUU \eta_{x_{i}}^2 \EEE \,dx.
	\end{align*}
	Choosing $\delta$ small enough and combining the above inequalities yield the desired inequality.
\end{proof}
{\GGG As for the standard orthotropic functional a key role is played by the following sophisticated
Caccioppoli-type inequality for the gradient, reminiscent of \cite[Proposition 3.2]{bblv}.}
\begin{prop}
\label{p:2.2}
\UUU For $\varepsilon \in (0, \varepsilon_0]$, let $u^\varepsilon$  be the minimizer of \eqref{eq:minepsilon} and assume that $a(\cdot)$ satisfies~\eqref{hpsobolev}. Then, for every $\theta \in (0, 2),$ $C_{0} >0,$ every $\Phi \in C^{1}(\R)$ monotone non-decreasing, \UUU  $\Psi \in C^{2} (\R^{+}; \R^{+})$ monotone non-decreasing, convex, with $\Psi(0) = 0$, and every Lipschitz function $\eta$ compactly supported in $B$, \GGG there exists a constant $C=C(n,p,q)>1$ such that \EEE 
	\begin{equation}\label{e:thesis-4.2}
	\begin{split}
		\sum_{i=1}^n\int &  \cM^{\varepsilon}_i(\nabla u^{\varepsilon} ) \RRR \Phi' \UUU(u^{\varepsilon}_{x_j})\Psi( ( u^{\varepsilon}_{x_k})^2) ( u^{\varepsilon}_{x_ix_j})^2 \eta^2  dx 
        \\
        &
        \le C_0 \sum_{i=1}^n \int \cM^{\varepsilon}_i ( \nabla u^{\varepsilon})\Phi^2(u^{\varepsilon}_{x_j}) \left( \Psi'( (u^{\varepsilon}_{x_k})^2 ) \right)^\theta (u^{\varepsilon}_{x_ix_j})^2 \eta^2 dx
        \\
		&
        \quad + C ( 1 + C_0^{-1} ) \sum_{i=1}^n \int \cM^{\varepsilon}_i(\nabla u^{\varepsilon})\left(\Psi(( u^{\varepsilon}_{x_k})^2 ) \right)^{2-\theta}|u^{\varepsilon}_{x_k}|^{2\theta}|\nabla \eta|^2dx
        \\
		&
        \quad + C \sum_{i=1}^n\int \cM^{\varepsilon}_i(\nabla u^{\varepsilon})\frac{\Phi^2(u^{\varepsilon}_{x_j})}{\Phi'(u^{\varepsilon}_{x_j})}\Psi( (u^{\varepsilon}_{x_k})^2)|\nabla \eta|^2  dx
        \\
        &
        \quad + C \sum_{i=1}^n\int |\nabla a|^2|u^{\varepsilon}_{x_i}|^{2q-p}\Phi'(u^{\varepsilon}_{x_j})\Psi( (u^{\varepsilon}_{x_k})^2)\eta^2 dx\\
		&
        \quad + C \sum_{i=1}^n\int |\nabla a |^2|u^{\varepsilon}_{x_i}|^{2q-p}\Phi^2(u^{\varepsilon}_{x_j})\left(\Psi'( (u^{\varepsilon}_{x_k})^2)\right)^\theta\eta^2 dx\\
		&
        \quad + C(1+C_0^{-1})\sum_{i=1}^n\int |\nabla a|^2|u^{\varepsilon}_{x_i}|^{2q-p}\Upsilon((u^{\varepsilon}_{x_k})^2) (u^{\varepsilon}_{x_k})^2\eta^2 dx,
	\end{split}
	\end{equation}
    \UUU where we have set for $t \in [0, +\infty)$
    \begin{equation}\label{e:Upsilon}
        \Upsilon(t)=4\left(\Psi'(t)\right)^{2-\theta}+\left(\Psi(t)\right)^{1-\frac{\theta}{2}}\left(\Psi'(t)\right)^{-\frac{\theta}{2}}|\sqrt{t}|^{\GGG \theta-2}\left(2\Psi'(t)+\Psi''(t) {\GGG t} \right).
    \end{equation} \EEE
\end{prop}

\begin{proof}
	\UUU Let us drop the index $\varepsilon$ from $u^\varepsilon$ and $\cM_i^\varepsilon$. We \EEE test \eqref{eq:ELAD} with $\varphi=  \Phi(u_{x_j}) \Psi(u_{x_k}^2)\eta^2$:
	\begin{equation}\label{e:1200}
	\begin{split}
		\sum_{i=1}^n&\int \cM_i(\nabla u)\Phi'(u_{x_j})\Psi(u_{x_k}^2)u_{x_ix_j}^2\eta^2\,dx\\
		&\GGG +q(q-2)\sigma_\varepsilon\int |\nabla u|^{q-2}\frac{\langle \nabla u,\nabla u_{x_j}\rangle^2}{|\nabla u|^2}\Phi'(u_{x_j})\Psi(u_{x_k}^2)\eta^2 dx\\
		&=-2\sum_{i=1}^n \int \cM_i(\nabla u)\Phi(u_{x_j})\Psi'(u_{x_k}^2)u_{x_k}u_{x_ix_j}u_{x_ix_k}\eta^2\,dx \\
		&-2\sum_{i=1}^n\int \cM_i(\nabla u)\Phi(u_{x_j})\Psi(u_{x_k}^2)u_{x_ix_j}\eta\eta_{x_i}\,dx \\
		&\GGG -2q(q-2)\sigma_\varepsilon\int |\nabla u|^{q-2}\langle\frac{\nabla u\otimes\nabla u}{|\nabla u|^2}\nabla u_{x_j}, \nabla u_{x_k}\rangle\Phi(u_{x_j})\Psi'(u^2_{x_k})u_{x_k}\eta^2dx\\
		&\GGG -2q(q-2)\sigma_\varepsilon\int |\nabla u|^{q-2}\langle\frac{\nabla u\otimes\nabla u}{|\nabla u|^2}\nabla u_{x_j}, \nabla \eta\rangle\Phi(u_{x_j})\Psi(u^2_{x_k})\eta\,dx\\
		& \quad-q\sum_{i=1}^n\int a_{x_j}|u_{x_i}|^{q-2}u_{x_i}\Phi'(u_{x_j})\Psi(u_{x_k}^2)u_{x_ix_j}\eta^2\,dx \\
		& \quad -2q\sum_{i=1}^n\int a_{x_j}|u_{x_i}|^{q-2}u_{x_i}\Phi(u_{x_j})\Psi'(u_{x_k}^2)u_{x_k}u_{x_ix_k}\eta^2\,dx\\
		& \quad -2q\sum_{i=1}^n\int a_{x_j}|u_{x_i}|^{q-2}u_{x_i}\Phi(u_{x_j})\Psi(u_{x_k})\eta\eta_{x_i}\,dx\\
		&=(\mathrm{I})+(\mathrm{II})+(\mathrm{III})+(\mathrm{IV})+(\mathrm{V})+\GGG (\mathrm{VI})+(\mathrm{VII}).
	\end{split}
	\end{equation}
	Let us estimate these terms. Using Young's inequality we have
	\begin{equation}\label{e:1201}
	\begin{split}
		(\mathrm{II})+\GGG (\mathrm{IV})\le& \GGG 2(q-1)\sum_{i=1}^n\int\cM_i(\nabla u)|\Phi(u_{x_j})|\Psi(u^2_{x_k})|u_{x_ix_j}||\eta||\nabla\eta|\,dx\\
		&\le\delta \sum_{i=1}^n\int \cM_i(\nabla u)\Phi'(u_{x_j})\Psi(u_{x_k}^2)u_{x_ix_j}^2\eta^2\,dx\\
		& +\frac{4(q-1)^2}{\delta}\sum_{i=1}^n\int \cM_i(\nabla u)\frac{\Phi^2(u_{x_j})}{\Phi'(u_{x_j})}\Psi(u_{x_k}^2)|\nabla\eta|^2\,dx,
		\end{split}
	\end{equation}
	for some small $\delta\in (0,1/10)$. Similarly we have
	\begin{equation}\label{e:1202}
	\begin{split}
		(\mathrm{V}) \le& \ \delta\sum_{i=1}^n\int \cM_i(\nabla u)\Phi'(u_{x_j})\Psi(u_{x_k}^2)u_{x_ix_j}^2\eta^2\,dx\\
		& + \frac{4q^2}{p(p-1)\delta}
        \sum_{i=1}^n\int  a_{x_j}^2|u_{x_i}|^{2q-p}\Phi'(u_{x_j})\Psi(u_{x_k}^2)\eta^2\,dx,
        \end{split}
        \end{equation}
        \begin{equation}\label{e:1203}
        \begin{split}
		(\mathrm{VII})\le&\ \sum_{i=1}^n\int \cM_i(\nabla u)\frac{\Phi^2(u_{x_j})}{\Phi'(u_{x_j})}\Psi(u_{x_k}^2)\eta_{x_i}^2\,dx\\
		& + \frac{q^2}{p(p-1)}
        \sum_{i=1}^n\int a_{x_j}^2|u_{x_i}|^{2q-p}\Phi'(u_{x_j})\Psi(u_{x_k}^2)\eta^2\,dx.
	\end{split}
	\end{equation}
     To deal with the remaining terms, \UUU for any $\theta\in(0,2)$ \EEE we write $\Psi'(u_{x_k}^2)=\left(\Psi'(u_{x_k}^2)\right)^{\frac{\theta}{2}}\left(\Psi'(u_{x_k}^2)\right)^{1-\frac{\theta}{2}}$ and use Cauchy-Schwartz inequality to obtain
	\begin{equation}\label{e:I-100}
	\begin{split}
		(\mathrm{I})\GGG + (\mathrm{III})\le & \ 2(q-1)\left(\sum_{i=1}^n\int \cM_i(\nabla u)\Phi^2(u_{x_j})\left(\Psi'(u_{x_k}^2)\right)^\theta u_{x_ix_j}^2\eta^2\,dx\right)^{\frac{1}{2}}\\
		&\times\left(\sum_{i=1}^n\int \cM_i(\nabla u)\left(\Psi'(u_{x_k}^2)\right)^{2-\theta}u_{x_k}^2u_{x_ix_k}^2\eta^2\,dx\right)^{\frac{1}{2}}\,, 
		\end{split}
		\end{equation}
		\begin{equation}\label{e:IV-100}
		\begin{split}
		(\mathrm{VI})\le& \ \frac{2q}{p(p-1)}\left(\sum_{i=1}^n\int a_{x_j}^2|u_{x_i}|^{2q-p}\Phi^2(u_{x_j})\left(\Psi'(u_{x_k}^2)\right)^\theta\eta^2\,dx\right)^{\frac{1}{2}}\\
		& \times\left(\sum_{i=1}^n\int \cM_i(\nabla u)\left(\Psi'(u_{x_k}^2)\right)^{2-\theta}u_{x_k}^2u_{x_ix_k}^2\eta^2\,dx\right)^{\frac{1}{2}}. 
	\end{split}
	\end{equation}
	\UUU We set $\Gamma(t):=\int_0^{t^2}\left(\Psi'(\tau)\right)^{1-\frac{\theta}{2}}\,d\tau$ and $\Lambda(t):=\Gamma(t)\Gamma'(t)$. We notice that 
    \begin{align}
            & \label{e:Gamma''}\Gamma''(t) = 2 ( \Psi'(t^{2}))^{1 - \frac{\theta}{2}} + 2t^{2} ( 2 - \theta ) (\Psi'(t^{2}))^{-\frac{\theta}{2}} \Psi''(t^{2}) \geq 0\,,\\
		&(\Gamma(t)\Gamma'(t))'=4t^2(\Psi'(t^2))^{2-\theta}+\Gamma(t)\Gamma''(t)\ge 4t^{2} (\Psi'(t^2))^{2-\theta}\,,\label{e:equation60}\\
            & \frac{\left(\Gamma(t)\Gamma'(t)\right)^2}{(\Gamma(t)\Gamma'(t))'}=\frac{\left(\Gamma(t)\Gamma'(t)\right)^2}{\Gamma(t)\Gamma''(t)+(\Gamma'(t))^2}\le\frac{\left(\Gamma(t)\Gamma'(t)\right)^2}{(\Gamma'(t))^2}=\Gamma^2(t)\,. \label{e:ineq-gamma}
	\end{align}
    In particular, $\Lambda \in C^{1} (\R)$ is non-decreasing. Applying Proposition \ref{prop:cac} and~\eqref{e:equation60}--\eqref{e:ineq-gamma} we obtain \EEE
	\begin{equation}\label{e:...} 
	\begin{split}
	\sum_{i=1}^n& \int \cM_i(\nabla u)\left(\Psi'(u_{x_k}^2)\right)^{2-\theta}u_{x_k}^2u_{x_ix_k}^2\eta^2\,dx \le C\sum_{i=1}^n\int \cM_i(\nabla u)\Gamma^2(u_{x_k})|\nabla \eta|^2\,dx\\
		&  +C\sum_{i=1}^n\int a_{x_k}^2|u_{x_i}|^{2q-p} \left( 4 \left(\Psi'(u_{x_k}^2)\right)^{2-\theta}u_{x_k}^2+\Gamma(u_{x_k})\Gamma''(u_{x_k})\right)\eta^2\,dx\,, 
	\end{split}
	\end{equation}
    where $C = C(p, q, n)>0$. Using Jensen inequality, \UUU we further estimate \EEE
	\begin{equation}\label{e:Gamma-100}
		\Gamma(t)=\int_0^{t^2}\left(\Psi'(\tau)\right)^{1-\frac{\theta}{2}}\,d\tau\le t^2\left(\frac{1}{t^2}\int_0^{t^2}\Psi'(\tau)\,d\tau\right)^{1-\frac{\theta}{2}}= |t|^\theta\Psi^{1-\frac{\theta}{2}}(t^2).
	\end{equation}
	Using the \UUU triangle and the weighted Young inequality we infer that for every $x_{1}, x_{2}, x_{3}, x_{4} \geq0$ it holds
	\begin{equation*}
		(x_{1}^{1/2} + x_{2}^{1/2})(x_{3}+ x_{4})^{1/2}\le (c_1 + c_2)x_{1} + (c_3 + c_4)x_{2} + \frac{(c_1+c_3)x_{3}}{4 c_1 c_3} + \frac{ ( c_2 + c_4)x_{4}}{4 c_2 c_4}
	\end{equation*}
	for arbitrary positive constants $c_1,c_2,c_3,c_4$. Combining~\eqref{e:I-100}--\eqref{e:Gamma-100} and properly choosing the constants $\{c_{i}\}_{i=1}^{4}$ above, we have, \GGG for some $C=C(n,p,q,\delta)\ge 1$,\EEE 
	\begin{equation}\label{e:I+IV}
	\begin{split}
		(\mathrm{I})+\GGG (\mathrm{III})+(\mathrm{VI})\le & \  (1-2\delta) \EEE C_0 \sum_{i=1}^n\int \cM_i(\nabla u)\Phi^2(u_{x_j})\left(\Psi'(u_{x_k}^2)\right)^\theta u_{x_ix_j}^2\eta^2\,dx\\
		& +C\sum_{i=1}^n\int |\nabla a|^2|u_{x_i}|^{2q-p}\Phi^2(u_{x_j})\left(\Psi'(u_{x_k}^2)\right)^\theta\eta^2\,dx\\
		& +C(1+C_0^{-1}) \sum_{i=1}^n\int \cM_i(\nabla u)\left(\Psi(u_{x_k}^2)\right)^{2-\theta}|u_{x_k}|^{2\theta}|\nabla \eta|^2 dx\\
		& +C(1+C_0^{-1}) \sum_{i=1}^n\int |\nabla a|^2|u_{x_i}|^{2q-p}\Upsilon(u_{x_k}^2)u_{x_k}^2\eta^2\,dx 
	\end{split}
	\end{equation}
	 where $\Upsilon$ is the function defined in~\eqref{e:Upsilon}.
	Combining~\eqref{e:I+IV} with~\eqref{e:1200}--\eqref{e:1203} and observing that the second term in the left hand-side of \eqref{e:1200} is nonnegative, we conclude for~\eqref{e:thesis-4.2}.
\end{proof}

{\GGG As in \cite[Proposition 4.1]{bblv}, applying Proposition~\ref{p:2.2} to suitable powers, we obtain the following result.}

\UUU
\begin{cor}
    \label{c:powers}
     For $\varepsilon \in (0, \varepsilon_0]$, let $u^\varepsilon$  be the minimizer of \eqref{eq:minepsilon} and assume that $a(\cdot)$ satisfies~\eqref{hpsobolev}. Then, there exists a positive constant~$C=C(n,p,q)$  such that for every integer $m \geq 1$, every $s \in [1, m]$, and every Lipschitz function $\eta$ compactly supported in $B$ it holds
     \begin{equation}\label{eq:cacbig}
     \begin{split}
		\sum_{i=1}^n\int & \cM^{\varepsilon}_i(\nabla u^{\varepsilon})|u^{\varepsilon}_{x_j}|^{2s-2}|u^{\varepsilon}_{x_k}|^{2m} (u^{\varepsilon}_{x_ix_j})^2 \eta^2 dx \\
		&\le \sum_{i=1}^n\int \cM^{\varepsilon}_i(\nabla u^{\varepsilon})|u^{\varepsilon}_{x_j}|^{4s-2}|u^{\varepsilon}_{x_k}|^{2m-2s} ( u^{\varepsilon}_{x_ix_j})^2 \eta^2 dx  \\
		& \quad +C\sum_{i=1}^n\int \cM_{i}^{\varepsilon} (\nabla u^{\varepsilon}) 
        |u^{\varepsilon}_{x_j}|^{2s+2m}|\nabla\eta|^2 dx  \\
		& \quad +C m \UUU \sum_{i=1}^n\int |\nabla a |^2|u^{\varepsilon}_{x_i}|^{2q-p}|u^{\varepsilon}_{x_j}|^{2m+2s-2}\eta^2 dx  \\
		& \quad +C(m+1)\sum_{i=1}^n\int \cM^{\varepsilon}_i(\nabla u^{\varepsilon})|u^{\varepsilon}_{x_k}|^{2m+2s}|\nabla\eta|^2 dx  \\
		& \quad +C(m^3+1)\sum_{i=1}^n\int |\nabla a |^2|u^{\varepsilon}_{x_i}|^{2q-p}|u^{\varepsilon}_{x_k}|^{2m+2s-2}\eta^2 dx. 
\end{split}
\end{equation}
\end{cor} \EEE

\begin{proof}
\UUU As usual, we drop the index $\varepsilon$ from $u^\varepsilon$ and $\cM_i^\varepsilon$. We apply Proposition~\ref{p:2.2} with \EEE $\Phi(t)=\frac{|t|^{2s-2}t}{2s-1}$, $\Psi(t)=t^m$ and
\begin{equation*}
	\theta=
	\begin{cases}
		\frac{m-s}{m-1}&\quad\text{if }m>1\\
		1&\quad\text{if }m=1\,.
	\end{cases}
\end{equation*}
Hence,~\eqref{e:thesis-4.2} becomes
\begin{align*}
	\sum_{i=1}^n\int & \cM_i(\nabla u)|u_{x_j}|^{2s-2}|u_{x_k}|^{2m}u_{x_ix_j}^2\eta^2 dx\\
	&\le \frac{C_0m^\theta}{(2s-1)^2} \sum_{i=1}^n\int \cM_i(\nabla u)|u_{x_j}|^{4s-2}|u_{x_k}|^{2m-2s} u_{x_ix_j}^2\eta^2 dx\\
	& \quad +C(1+C_0^{-1}) \sum_{i=1}^n\int \cM_i(\nabla u)|u_{x_k}|^{2m+2s}|\nabla\eta|^2dx\\
	& \quad +\frac{C}{(2s-1)^2} \sum_{i=1}^n\int \cM_i(\nabla u)|u_{x_j}|^{2s}|u_{x_k}|^{2m}|\nabla \eta|^2 dx\\
	& \quad +C \sum_{i=1}^n\int |\nabla a|^2|u_{x_i}|^{2q-p}|u_{x_j}|^{2s-2}|u_{x_k}|^{2m}\eta^2 dx\\
	& \quad +\frac{C \UUU m^{\theta} \EEE }{(2s-1)^2} \sum_{i=1}^n\int |\nabla a|^2|u_{x_i}|^{2q-p}|u_{x_j}|^{4s-2}|u_{x_k}|^{2m-2s}\eta^2 dx\\
	& \quad +7C(1+C_0^{-1})m^2 \sum_{i=1}^n\int |\nabla a|^2|u_{x_i}|^{2q-p}|u_{x_k}|^{2m+2s-2}\eta^2 dx.
\end{align*}
Choosing $C_0=m^{-1}$ gives, after simplifying some inequalities
\begin{align*}
	\sum_{i=1}^n\int & \cM_i(\nabla u)|u_{x_j}|^{2s-2}|u_{x_k}|^{2m}u_{x_ix_j}^2\eta^2 dx\\
	& \le \sum_{i=1}^n \int\cM_i(\nabla u)|u_{x_j}|^{4s-2}|u_{x_k}|^{2m-2s} u_{x_ix_j}^2\eta^2 dx\\
	& \quad +C(m+1)\sum_{i=1}^n \int \cM_i(\nabla u)|u_{x_k}|^{2m+2s}|\nabla \eta|^2 dx\\
	& \quad +C\sum_{i=1}^n \int \cM_i(\nabla u)|u_{x_j}|^{2s}|u_{x_k}|^{2m}|\nabla\eta|^2 dx\\
	& \quad +C\sum_{i=1}^n \int |\nabla a|^2|u_{x_i}|^{2q-p}|u_{x_j}|^{2s-2}|u_{x_k}|^{2m}\eta^2 dx\\
	& \quad +C\UUU m^{\theta} \EEE \sum_{i=1}^n\int |\nabla a|^2|u_{x_i}|^{2q-p}|u_{x_j}|^{4s-2}|u_{x_k}|^{2m-2s}\eta^2 dx\\
	& \quad +C(m^3+1) \sum_{i=1}^n \int |\nabla a|^2|u_{x_i}|^{2q-p}|u_{x_k}|^{2m+2s-2}\eta^2 dx\,.
\end{align*}
Using Young inequalities with the right weights gives~\eqref{eq:cacbig}.
\end{proof}

\begin{prop}
\label{p:iteration}
	 For $\varepsilon \in (0, \varepsilon_0]$, let $u^\varepsilon$  be the minimizer of \eqref{eq:minepsilon} and assume that $a(\cdot)$ satisfies~\eqref{hpsobolev}. Let~$\ell_0$ be a positive integer and set $\sigma: = 2^{\ell_0}-1$. \EEE Then there exists a positive constant $C=C(n,p,q)$ such that \UUU for every Lipschitz function~$\eta$ compactly supported in~$B$ \EEE
	\begin{equation}\label{eq:estsigma}
	\begin{split}
			\int & \left(\left|\nabla\left(|u^\varepsilon_{x_k}|^{\sigma+\frac{p}{2}}\right)\right|^2 + a \left|\nabla\left(|u^\varepsilon_{x_k}|^{\UUU \sigma+\frac{q}{2} \EEE}\right)\right|^2\right)\eta^2 dx\\
			& \le C\sigma^3 \sum_{i,j=1}^n\int \cM^\varepsilon_{i} (\nabla u^{\varepsilon}) |u^\varepsilon_{x_j}|^{2\sigma+2}|\nabla \eta|^2 dx  + \UUU C \sigma^4 \EEE \sum_{i,j=1}^n\int |\nabla a|^2|u^\varepsilon_{x_i}|^{2q-p}|u^\varepsilon_{x_j}|^{2\sigma}\eta^2 dx  \\
			& \quad +C\sigma^4 \sum_{i=1}^n\int \cM^\varepsilon_{i} (\nabla u^{\varepsilon}) |u^\varepsilon_{x_k}|^{2\sigma+2}|\nabla \eta|^2 dx  +C\sigma^6 \sum_{i=1}^n\int |\nabla a |^2|u^\varepsilon_{x_i}|^{2q-p}|u^\varepsilon_{x_k}|^{2\sigma}\eta^2 dx. 
	\end{split}
	\end{equation}
\end{prop}

\begin{proof}
	\UUU We drop the index $\varepsilon$. \EEE For $\ell=0,1,\dots,\ell_0$ we take
	\begin{equation*}
		s_\ell=2^\ell,\quad m_\ell=2^{\ell_0}-2^\ell=\sigma+1-2^\ell.
	\end{equation*}
	These families satisfy
	\begin{align*}
		& s_0=1, \ m_0=\sigma, \ s_{\ell_0}=\sigma+1, \ m_{\ell_0}=0\,,\\
		& 1\le s_\ell\le m_\ell\quad\forall\ell=0,\dots,\ell_0-1\,,\\
		& s_\ell+m_\ell=\sigma+1\quad\forall\ell=0,\dots,\ell_0\,,\\
		& 4s_\ell-2=2s_{\ell+1}-2,\quad 2m_\ell-2s_\ell=2m_{\ell+1}\quad\forall\ell=0,\dots,\ell_0-1.
	\end{align*}
	Substituting in \eqref{eq:cacbig} gives
	\begin{align*}
		\sum_{i=1}^n\int & \cM_i(\nabla u)|u_{x_j}|^{2s_\ell-2}|u_{x_k}|^{2m_\ell}u_{x_ix_j}^2\eta^2 dx\\
		& \le \sum_{i=1}^n\int \cM_i(\nabla u)|u_{x_j}|^{2s_{\ell+1}-2}|u_{x_k}|^{2m_{\ell+1}}u_{x_ix_j}^2\eta^2 dx\\
		& \quad +C\sum_{i=1}^n\int \cM_i(\nabla u)|u_{x_j}|^{2\sigma+2}|\nabla\eta|^2 dx +C \UUU m_{\ell} \EEE \sum_{i=1}^n\int |\nabla a|^2|u_{x_i}|^{2q-p}|u_{x_j}|^{2\sigma}\eta^2 dx\\
		& \quad +C(m_\ell+1) \sum_{i=1}^n\int \cM_i(\nabla u)|u_{x_k}|^{2\sigma+2}|\nabla\eta|^2 dx\\
		& \quad +C(m_\ell^3+1) \sum_{i=1}^n\int |\nabla a|^2|u_{x_i}|^{2q-p}|u_{x_k}|^{2\sigma}\eta^2 dx\,.
	\end{align*}
	Iterating from $\ell=0$ to $\ell=\ell_0-1$ gives
	\begin{equation}\label{e:iterate}
	\begin{split}
		\sum_{i=1}^n\int & \cM_i(\nabla u)|u_{x_k}|^{2\sigma}u_{x_ix_j}^2\eta^2 dx\\
		& \le \sum_{i=1}^n\int \cM_i(\nabla u)|u_{x_j}|^{2\sigma}u_{x_ix_j}^2\eta^2 dx+C\ell_0 \sum_{i=1}^n\int \cM_i(\nabla u)|u_{x_j}|^{2\sigma+2}|\nabla \eta|^2 dx r\\
		& \quad  +C\ell_0 \UUU \sigma \EEE  \sum_{i=1}^n\int |\nabla a|^2|u_{x_i}|^{2q-p}|u_{x_j}|^{2\sigma}\eta^2 dx +C\ell_0\sigma \sum_{i=1}^n\int \cM_i(\nabla u)|u_{x_k}|^{2\sigma+2}|\nabla \eta|^2 dx\\
		& \quad  +C\ell_0\sigma^3 \sum_{i=1}^n\int |\nabla a|^2|u_{x_i}|^{2q-p}|u_{x_k}|^{2\sigma}\eta^2 dx. 
	\end{split}
	\end{equation}
	We now use \eqref{eq:caca} with $\UUU \Lambda (t) \EEE =\frac{|t|^{2\sigma}t}{2\sigma+1}$ to estimate
	\begin{align*}
		\sum_{i=1}^n\int & \cM_i(\nabla u)|u_{x_j}|^{2\sigma}u_{x_ix_j}^2\eta^2 dx\\
		& \le\frac{C}{(2\sigma+1)^2}\sum_{i=1}^n\int \cM_i(\nabla u)|u_{x_j}|^{2\sigma+2}|\nabla \eta|^2 dx
		+\UUU C \EEE \sum_{i=1}^n\int |\nabla a|^2|u_{x_i}|^{2q-p}|u_{x_j}|^{2\sigma}\eta^2 dx\\
		& \le C\sum_{i=1}^n\int \cM_i(\nabla u)|u_{x_j}|^{2\sigma+2}|\nabla \eta|^2 dx
		+ \UUU C \EEE  \sum_{i=1}^n\int |\nabla a|^2|u_{x_i}|^{2q-p}|u_{x_j}|^{2\sigma}\eta^2 dx\,.
	\end{align*}
	Combining the above inequality with~\eqref{e:iterate} we get
	\begin{equation}\label{e:iterate-2}
	\begin{split}
		\sum_{i=1}^n&\int  \cM_i(\nabla u)|u_{x_k}|^{2\sigma}u_{x_ix_k}^2\eta^2 dx\\
		& \le C\ell_0 \sum_{i=1}^n\int \cM_i(\nabla u)|u_{x_j}|^{2\sigma+2}|\nabla \eta|^2\,dx +C\ell_0 \UUU \sigma \EEE \sum_{i=1}^n\int |\nabla a|^2|u_{x_i}|^{2q-p}|u_{x_j}|^{2\sigma}\eta^2 dx \\
		& +C\ell_0\sigma\sum_{i=1}^n\int \cM_i(\nabla u)|u_{x_k}|^{2\sigma+2}|\nabla \eta|^2 dx + C\ell_0\sigma^3 \sum_{i=1}^n\int |\nabla a|^2|u_{x_i}|^{2q-p}|u_{x_k}|^{2\sigma}\eta^2 dx.
		\end{split}
	\end{equation}
	Notice that for $\gamma = p, q$ we have \EEE
	\begin{equation*}
		\sum_{i=1}^n|u_{x_i}|^{\UUU \gamma \EEE -2}|u_{x_k}|^{2\sigma}u_{x_ix_j}^2 \UUU \geq \EEE |u_{x_k}|^{\UUU \gamma \EEE -2}|u_{x_k}|^{2\sigma}u_{x_kx_j}^2 \UUU = \EEE \left(\sigma+\frac{\UUU \gamma \EEE }{2}\right)^{-2}\left|\left(|u_{x_k}|^{\sigma+\frac{\UUU \gamma \EEE }{2}}\right)_{x_j}\right|^2.
	\end{equation*}
	 By excluding the nonnegative term with $|\nabla u|^{q-2}$ from $\cM_i$ in the left hand side of~\eqref{e:iterate-2}, since $\ell_0\le\sigma$ we obtain \EEE
	\begin{align*}
		\int & \left(\left|\left(|u_{x_k}|^{\sigma+\frac{p}{2}}\right)_{x_j}\right|^2+a\left|\left(|u_{x_k}|^{\sigma+ \UUU \frac{q}{2} \EEE}\right)_{x_j}\right|^2\right)\eta^2 dx\\
		& \le C\sigma^3\sum_{i=1}^n \int \cM_i(\nabla u)|u_{x_j}|^{2\sigma+2}|\nabla \eta|^2 dx +C \UUU \sigma^4 \EEE \sum_{i=1}^n \int |\nabla a|^2|u_{x_i}|^{2q-p}|u_{x_j}|^{2\sigma}\eta^2 dx 
        \\
        &\quad +C\sigma^4\sum_{i=1}^n \int \cM_i(\nabla u)|u_{x_k}|^{2\sigma+2}|\nabla \eta|^2 dx +C\sigma^6\sum_{i=1}^n \int |\nabla a|^2|u_{x_i}|^{2q-p}|u_{x_k}|^{2\sigma}\eta^2 dx.
	\end{align*}
	 \UUU Summing up over $j$ gives \eqref{eq:estsigma}, and the proof is thus concluded. \EEE
\end{proof}
\section{Proof of  Theorem~\ref{thm:main}}
\subsection{The key estimate}

The core of the proof is the following Lipschitz estimate. The proof is very similar to that of \cite[Theorem 5.1]{bblv}, though some important technical modifications have to be taken into account.
\begin{prop}\label{prop:lip}
	\UUU For $\varepsilon \in (0, \varepsilon_0]$, let $u^\varepsilon$  be the minimizer of \eqref{eq:minepsilon} and assume that $a( \cdot)$ satisfies~\eqref{hpsobolev}. Then, there exist positive constants $C,\mu>0$ and $\Theta\geq 1$ (depending  on~$n,p,q,a$) \EEE such that for every $0<r<R<1$ and every concentric balls $B_r\subset B_R$  \UUU it holds \EEE
	\begin{equation}\label{eq:lip}
		\|\nabla u^\varepsilon\|_{L^\infty(B_r)}\le \frac{C}{(R-r)^{\mu}}\left(\int_{B_R}|\nabla u^\varepsilon|^{p}+1\right)^{ \frac{\Theta}{p} }\,.
	\end{equation}
\end{prop}
\begin{proof}
	We prove the proposition in the case $n\ge 3$, so that $2^*$ is finite. An analogous argument works in dimension $n=2$ replacing $2^*$ with a sufficiently large exponent.
    
    As usual, we drop the superscript $\varepsilon$ for simplicity. \UUU For any Lipschitz continuous function $\eta$ compactly supported in $B=B_R$, we add to both sides of \eqref{eq:estsigma} (skipping the nonnegative second term in the left hand side) the term $\int |\nabla \eta|^2|u_{x_k}|^{p+2\sigma}\,dx$, obtaining \EEE
	\begin{align*}
		&\int \left|\nabla\left(|u_{x_k}|^{\sigma+\frac{p-2}{2}}u_{x_k}\eta\right)\right|^2 dx
        \\
        &
	\qquad 	\le  \ C\sigma^3\sum_{i,j=1}^n \int \cM_i(\nabla u)|u_{x_j}|^{2\sigma+2}|\nabla\eta|^2 dx +C \UUU \sigma^4 \EEE \sum_{i,j=1}^n\int |\nabla a|^2|u_{x_i}|^{2q-p}|u_{x_j}|^{2\sigma}\eta^2 dx\\
		& \qquad \quad +C\sigma^4\sum_{i=1}^n \int \cM_i(\nabla u)|u_{x_k}|^{2\sigma+2}|\nabla \eta|^2 dx +C\sigma^6\sum_{i=1}^n \int |\nabla a|^2|u_{x_i}|^{2q-p}|u_{x_k}|^{2\sigma}\eta^2 dx\\
		& \qquad \quad +C \int |\nabla\eta|^2|u_{x_k}|^{p+2\sigma}\,dx.
	\end{align*}
	 \EEE We apply Sobolev inequality on the left hand side:
	\begin{align*}
		& \left(\int |u_{x_k}|^{\frac{2^*}{2}(p+2\sigma)}\eta^{2^*}dx\right)^{2/2^*}
		\\
        & \qquad \le  \ C\sigma^3\sum_{i,j=1}^n \int \cM_i(\nabla u)|u_{x_j}|^{2\sigma+2}|\nabla\eta|^2 dx +C \UUU \sigma^4 \EEE \sum_{i,j=1}^n \int |\nabla a|^2|u_{x_i}|^{2q-p}|u_{x_j}|^{2\sigma}\eta^2 dx\\
		& \qquad \quad +C\sigma^4\sum_{i=1}^n \int \cM_i(\nabla u)|u_{x_k}|^{2\sigma+2}|\nabla\eta|^2 dx+C\sigma^6\sum_{i=1}^n \int |\nabla a|^2|u_{x_i}|^{2q-p}|u_{x_k}|^{2\sigma}\eta^2 dx\\
		& \qquad \quad  +C \int |\nabla\eta|^2|u_{x_k}|^{p+2\sigma}dx.
	\end{align*}
	Let us now sum over $k=1,\dots,n$. Minkowski inequality implies that
	\begin{equation*}
		\sum_{k=1}^n\left(\int |u_{x_k}|^{\frac{2^*}{2}(p+2\sigma)}\eta^{2^*} dx\right)^{2/2^*}=\sum_{k=1}^n\left\||u_{x_k}|^{p+2\sigma}\eta^2\right\|_{L^{2^*/2}}\ge\left\|\sum_{k=1}^n|u_{x_k}|^{p+2\sigma}\eta^2\right\|_{L^{2^*/2}}
	\end{equation*}
	so that (recall that $\sigma>1$):
	\begin{align*}
		& \left(\int \left(\sum_{k=1}^n|u_{x_k}|^{(p+2\sigma)}\right)^{\frac{2^*}{2}}\eta^{2^*}dx\right)^{\frac{2}{2^*}}\\
		& \qquad \le C\sigma^4\sum_{i,k=1}^n \int \cM_i(\nabla u)|u_{x_k}|^{2\sigma+2}|\nabla\eta|^2 dx +C\sigma^6\sum_{i,k=1}^n \int |\nabla a|^2|u_{x_i}|^{2q-p}|u_{x_k}|^{2\sigma}\eta^2 dx\\
		& \qquad \quad +C\sum_{k=1}^n \int |\nabla\eta|^2|u_{x_k}|^{p+2\sigma}dx.
	\end{align*}
	We define
	\begin{equation*}
		\cU(x):=\max_{k=1,\dots,n}|u_{x_k}(x)|
	\end{equation*}
	which satisfies, \UUU for every $\alpha >0,$ \EEE
	\begin{equation*}
    		\cU^\alpha\le\sum_{k=1}^n|u_{x_k}|^\alpha\le n\cU^\alpha \qquad \text{and} \qquad 
		|\nabla u|^\alpha\le n^{\alpha/2}\cU^\alpha.
	\end{equation*}
	Using this function \UUU and recalling the definition of~$\cM_{i} (\nabla u)$ (cf.~\eqref{e:Miepsilon}) \EEE we have
	\begin{equation}\label{e:U1}
	\begin{split}
		\left(\int \cU^{\frac{2^*}{2}(p+2\sigma)}\eta^{2^*}dx\right)^{\frac{2}{2^*}}
		&\le \ C\sigma^4 \int \cU^{p+2\sigma}|\nabla\eta|^2  dx \\
		& +C\sigma^4 \int (a(x)+1)\cU^{q+2\sigma}|\nabla\eta|^2 dx +C\sigma^6 \int |\nabla a|^2\cU^{2q-p+2\sigma}\eta^2 dx. 
	\end{split}
	\end{equation}
	 \UUU Let $0 < r < \rho< R< 1$ and let \UUU $B_s \subset B_t$ \EEE be two concentric balls with $r\le s<t\le \rho$. Let us choose \EEE $\eta=\eta_{s,t}$ such that $\textrm{supp}\UUU (\eta) \EEE =B_t$, \UUU $0 \leq \eta \leq 1$, $\eta =  1$ on $B_{s}$, \EEE and $\|\nabla\eta\|_{L^\infty}\le\frac{C}{t-s}$. \UUU By H\"older inequality, we continue in~\eqref{e:U1} with \EEE
     \begin{align*}
		 \left(\int_{B_s} \cU^{\frac{2^*}{2}(p+2\sigma)}\,dx\right)^{\frac{2}{2^*}}
		\le & \  \frac{C\sigma^4}{(t-s)^2} \int_{B_t} \cU^{p+2\sigma}\,dx+\frac{C\sigma^4(\|a\|_{L^\infty}+1)}{(t-s)^2} \int_{B_t} \cU^{q+2\sigma}\,dx\\
		& +C\sigma^6\|\nabla a\|_{L^{\frac{n}{1-\alpha}}}^2 \left(\int_{B_t} \cU^{(2q-p+2\sigma)h}\eta^2\,dx\right)^{\frac{1}{h}},
	\end{align*}
	where \UUU we have set \EEE $h: =\frac{n}{n-2(1-\alpha)}$. \UUU We further simplify the above expression with \EEE
	\begin{equation}\label{eq:hadh}
		\left(\int_{B_s} \cU^{\frac{2^*}{2}(p+2\sigma)}\,dx\right)^{\frac{2}{2^*}}\le \frac{C\sigma^6}{(t-s)^2}\left(1+\|a\|_{L^\infty}+\|\nabla a\|_{L^{\frac{n}{1-\alpha}}}^2\right)\left(\int_{B_t}\cU^{(2q-p+2\sigma)h}\,dx+1\right)^{\frac{1}{h}}.
	\end{equation}
     
        \EEE For every $j \in \mathbb{N},$ let us consider the following quantities
        \begin{align}
            \label{eq:sig}
		\sigma_j & :=2^{j+1}-1\,,\\
            \label{eq:gam}
		\gamma_j & :=(2q-p+2 \UUU \sigma_{j} \EEE )h=\frac{n(2q-p+2^{j+2}-2)}{n-2+2\alpha}\,,\\
            \hat{\gamma}_j & :=\frac{2^*}{2}(p+2 \UUU \sigma_{j} \EEE )=\frac{n(p+2^{j+2}-2)}{n-2}\,.\label{eq:gam-2}
	\end{align}
	\UUU We thus rewrite~\eqref{eq:hadh}:
	 \EEE
	\begin{equation}\label{eq:hse}
		\left(\int_{B_s} \cU^{\hat{\gamma}_j}\,dx\right)^{\frac{2}{2^*}}\le \frac{C2^{6j}}{(t-s)^2}\left(1+\|a\|_{L^\infty}+\|\nabla a\|_{L^{\frac{n}{1-\alpha}}}^2\right)\left(\int_{B_t}\cU^{\gamma_j}\,dx+1\right)^{\frac{1}{h}}.
	\end{equation}
	
        We now use the interpolation inequality in Lebesgue spaces to deduce an inequality involving $L^{\gamma_j}$ and $L^{\gamma_{j-1}}$ norms. First, we notice that
	\begin{equation*}
		\hat{\gamma}_j>\gamma_j \ \Longleftrightarrow \ h(2q-p-2 \UUU + \EEE 2^{j+2})<\frac{2^*}{2}(p-2+2^{j+2}) \ \Longleftrightarrow \  2^{j+2}>\frac{n-2}{\alpha}(q-p)-(p-2)\,.
	\end{equation*}
	\UUU Since \EEE
	\begin{equation*}
		\frac{q}{p}<1+\frac{\alpha}{n}<1+\frac{\alpha}{n-2},
	\end{equation*}
	we \UUU immediately deduce that \EEE
	\begin{equation}\label{graziehp}
		\frac{n-2}{\alpha}(q-p)-(p-2)<2
	\end{equation}
	and thus $\gamma_{j-1}<\gamma_j<\hat{\gamma}_j$ for all $j\ge 0$.
	We can therefore consider the interpolation inequality
	\begin{equation}\label{e:interpolation-ineq}
		\int_{B_s}\cU^{\gamma_j}\,dx\le\left(\int_{B_s}\cU^{\gamma_{j-1}}\,dx\right)^{\frac{\tau_j\gamma_j}{\gamma_{j-1}}}\left(\int_{B_s}\cU^{\hat{\gamma}_j}\,dx\right)^{\frac{(1-\tau_j)\gamma_j}{\hat{\gamma}_j}}
	\end{equation}
	where the value of  \UUU $\tau_j \in (0, 1)$ \EEE is
	\begin{equation}\label{e:tauj}
		\tau_j : =\frac{\displaystyle\frac{\hat{\gamma}_j}{\gamma_j}-1}{\displaystyle\frac{\hat{\gamma}_j}{\gamma_j}\cdot\frac{\gamma_j}{\gamma_{j-1}}-1}.
	\end{equation}
	\GGG The sequence $\{\tau_j\}_{j\in\mathbb{N}}$ satisfies $\displaystyle \lim_{j\to+\infty}\tau_j=\tau_\infty\in (0,1)$.
	
	\UUU Let us further consider the positive quantity
	\begin{equation*}
		M_j:=(1-\tau_j)\frac{2^*}{2h}\cdot\frac{\gamma_j}{\hat{\gamma}_j}=(1-\tau_j)\frac{2q-p+2\sigma_j}{p+2\sigma_j}.
	\end{equation*}
	We can write
	\begin{equation*}
		M_j=\frac{\frac{\hat{\gamma}_j}{\gamma_{j-1}}-\frac{\hat{\gamma}_j}{\gamma_j}}{\frac{\hat{\gamma}_j}{\gamma_{j-1}}-1}\cdot\frac{\gamma_j}{\hat{\gamma}_j}\cdot\frac{2^*}{2h}=\frac{\gamma_j-\gamma_{j-1}}{\hat{\gamma}_j-\gamma_{j-1}}\cdot\frac{2^*}{2h}=\frac{2^{j+1}}{\frac{2^*}{2h}(p-2+2^{j+2})-(2q-p-2\GGG +\UUU 2^{j+1})}\cdot\frac{2^*}{2h}
	\end{equation*}
	so that $M_j<1$ if
	\begin{equation*}
		\frac{2^*}{2h}(p-2+2^{j+2})-(2q-p-2\GGG +\UUU 2^{j+1})>2^{j+1}\cdot\frac{2^*}{2h}
	\end{equation*}
	and therefore
	\begin{equation*}
		M_j<1 \ \Longleftrightarrow \ 2^{j+1}>\frac{n-2}{\alpha}(q-p)-(p-2).
	\end{equation*}
	Since, for every $j\ge 0$, the right hand side is smaller than $2$, \GGG as we already observed in \eqref{graziehp}, and since $M_j\to 1-\tau_\infty\in(0,1)$, \UUU we have that there exist two constants $C_1,C_2$ only depending on~$p$, $q$, $\alpha$, and $n$, such that
	\begin{equation}\label{e:boundMj}
		0<C_1\le M_j\le C_2<1\quad\forall j\ge 0.
	\end{equation}
    Applying~\eqref{eq:hse} to~\eqref{e:interpolation-ineq}, using Young inequality and~\eqref{e:boundMj} we infer that \EEE
	\begin{equation}\label{e:2000}
	\begin{split}
		\int_{B_s} \cU^{\gamma_j}\,dx&\le \left(\frac{C2^{6j}}{(t-s)^2}\right)^{\frac{2^*(1-\tau_j)\gamma_j}{2\hat{\gamma}_j}}\left(\int_{B_s}\cU^{\gamma_{j-1}}\,dx\right)^{\frac{\tau_j\gamma_j}{\gamma_{j-1}}}\left(\int_{B_t}\cU^{\gamma_j}\,dx+1\right)^{\frac{2^*}{2h}\cdot\frac{(1-\tau_j)\gamma_j}{\hat{\gamma}_j}}  \\
		& \UUU = \EEE \left(\left(\frac{C2^{6j}}{(t-s)^2}\right)^{\frac{2^*(1-\tau_j)\gamma_j}{2\tau_j\hat{\gamma}_j}}\left(\int_{B_s}\cU^{\gamma_{j-1}}\,dx\right)^{\frac{\gamma_j}{\gamma_{j-1}}}\right)^{\tau_j}\left(\int_{B_t}\cU^{\gamma_j}\,dx+1\right)^{M_j}
        \\
        &
        \le C_2\left(\int_{B_t}\cU^{\gamma_j}\,dx+1\right)+\left(1-C_1\right)\left(\frac{C2^{6j}c_a}{(t-s)^2}\right)^{\beta}\left(\int_{B_s}\cU^{\gamma_{j-1}}\,dx\right)^{\frac{\tau_j\gamma_j}{\gamma_{j-1}}\cdot\left(\frac{1}{M_j}\right)'} 
	\end{split}
	\end{equation}
	where $C=C(n,p,q,a)$ and $\beta\in(1,+\infty)$ is a suitable exponent dependent on $p,q,\alpha,n$.
	
	 \UUU Defining \EEE
	\begin{equation*}
		\UUU \kappa_{j} \EEE :=\tau_j\left(\frac{1}{M_j}\right)'-1
	\end{equation*}
	we \UUU rewrite~\eqref{e:2000} as \EEE
	\begin{equation*}
		\int_{B_s} \cU^{\gamma_j}\,dx\le C_2\int_{B_t}\cU^{\gamma_j}\,dx+\left(1-C_1\right)\left(\frac{C2^{6j}}{(t-s)^2}\right)^{\beta}\left(\int_{B_s}\cU^{\gamma_{j-1}}\,dx\right)^{\frac{\gamma_j}{\gamma_{j-1}}(1+ \UUU \kappa_j \EEE)}+ \UUU C_2. \EEE
	\end{equation*}
	 We now apply Lemma \ref{lem:est} with
	\begin{align*}
		& Z(y)=\int_{B_y}\cU^{\gamma_j}\,dx,\quad \UUU \cA = \GGG (1-C_1)\UUU \big( C2^{6j} \big)^{\beta}\left(\int_{B_{\rho}} \cU^{\gamma_{j-1}}\,dx\right)^{\frac{\gamma_j}{\gamma_{j-1}}(1+ \UUU \kappa_j \EEE)}\,,  \EEE 
        \\
        &
        \UUU \cB=0\,, \quad \mathcal{C} = C_{2}\,, \EEE  \quad\alpha_0=2\beta\,, \quad\theta=C_2\,.
	\end{align*}
	\UUU Hence, we obtain that, \EEE for all $j\ge 0,$ 
	\begin{equation}\label{eq:mos0}
		\int_{\UUU B_{r} \EEE } \cU^{\gamma_j}\,dx\le C\left(\left(\frac{2^{6j}}{ \UUU (\rho - r)^2 \EEE }\right)^{\beta}\left(\int_{\UUU B_{\rho} \EEE }\cU^{\gamma_{j-1}}\,dx\right)^{\frac{\gamma_j}{\gamma_{j-1}}(1+\UUU \kappa_{j} \EEE )}+1\right)
	\end{equation}
	with $C=C(n,p,q,a)>1$. 
    
    \GGG We now want to iterate the previous estimate on a sequence of shrinking balls. We fix two radii $0 < r < \rho\le 1$, then we consider the sequence
    $$
    R_j=r+\frac{\rho-r}{2^j}
    $$
and we apply \eqref{eq:mos0} with $R_{j+1}< R_j$ instead of $r < \rho$.  \UUU For later use, we define 
    \begin{equation}\label{e:lambdaj}
		\lambda_j:=\frac{2h}{2^*}\frac{\hat{\gamma}_j}{\gamma_j}=\frac{p+2 \UUU \sigma_{j} \EEE }{2q-p+2 \UUU \sigma_{j} \EEE } \quad \in (0, 1]
	\end{equation}
	and notice that we may rewrite
	\begin{equation}\label{e:kappaj}
	\begin{split}
		\UUU \kappa_{j} \EEE & =-1+\tau_j\cdot\left(\frac{2h}{2^*}\frac{1}{1-\tau_j}\frac{\hat{\gamma}_j}{\gamma_j}\right)'=-1+\tau_j\cdot\left(\frac{\UUU \lambda_{j} \EEE }{1-\tau_j}\right)'
        \\
        &
        =-1+\frac{\tau_j}{1-\frac{1-\tau_j}{\UUU \lambda_{j} \EEE }}=\frac{(1- \UUU \lambda_{j} \EEE )(1-\tau_j)}{ \UUU \lambda_{j} \EEE +\tau_j-1}\,. 
	\end{split}
	\end{equation}
	In particular, $\kappa_{j}$ is nonnegative for $j\ge 0$, as $\gamma_{j-1} < \gamma_{j} < \hat{\gamma}_{j}$ and $\frac{2h}{2^{*}} >1$.
	To simplify our notation, we denote
	\begin{equation*}
		Y_j:=\int_{B_{R_j}}\cU^{\gamma_{j-1}}\,dx
	\end{equation*}
	so that \UUU \eqref{eq:mos0} writes as \EEE
	\begin{equation*}
		Y_{j+1}\le C\left(2^{8\beta j}(\UUU \rho \EEE -r)^{-2\beta} \, Y_j^{\frac{\gamma_j}{\gamma_{j-1}}(1+\UUU \kappa_{j} \EEE ) } + 1 \right)\le\left(C2^{8\beta}(\UUU \rho \EEE -r)^{-2\beta}\right)^j\left(Y_j+1\right)^{\frac{\gamma_j}{\gamma_{j-1}}(1+ \UUU \kappa_{j} \EEE)}
	\end{equation*}
	since $C>1$ and $\rho\le 1$. Iterating this inequality gives
	\begin{align*}
		Y_{j+1}&\le\left(C2^{8\beta}( \UUU \rho \EEE -r)^{-2\beta}\right)^j\left(Y_j+1\right)^{\frac{\gamma_j}{\gamma_{j-1}}(1+ \UUU \kappa_{j} \EEE)}
        \\
		&
        \le \left(C2^{8\beta}( \UUU \rho \EEE -r)^{-2\beta}\right)^j\left(\left(C2^{8\beta}( \UUU \rho \EEE -r)^{-2\beta}\right)^{j-1}\left(Y_{j-1}+1\right)^{\frac{\gamma_{j-1}}{\gamma_{j-2}}(1+ \UUU \kappa_{j-1} \EEE ) } + 1 \right)^{\frac{\gamma_j}{\gamma_{j-1}}(1+\UUU \kappa_{j} \EEE )}\\
		&\le\left(C2^{8\beta}( \UUU \rho \EEE -r)^{-2\beta}\right)^j\left(2\left(C2^{8\beta}(R-r)^{-2\beta}\right)^{j-1}\left(Y_{j-1}+1\right)^{\frac{\gamma_{j-1}}{\gamma_{j-2}}(1+\UUU \kappa_{j-1} \EEE )}\right)^{\frac{\gamma_j}{\gamma_{j-1}}(1+ \UUU \kappa_{j} \EEE )}.
	\end{align*}
	Therefore, we infer that
	\begin{align}\label{e:YN+1}
		Y_{N+1}\le\left(C2^{8\beta+1}( \UUU \rho \EEE -r)^{-2\beta}\right)^{\sum_{j=1}^{N}\left(j\frac{\gamma_N}{\gamma_j} \UUU \prod_{k=j}^N(1+\kappa_k) \EEE\right)}\left(Y_1+1\right)^{\frac{\gamma_N}{\gamma_0}\prod_{j=1}^{N}(1+ \UUU \kappa_{j} \EEE)}. 
	\end{align}
	\UUU Denoting the constant~$C2^{8\beta+1}$ again \EEE with $C$ and setting 
         \begin{displaymath}
            \Theta_N:=\prod_{j=1}^N(1+\kappa_j)\,,
         \end{displaymath} 
         we take the power $1/\gamma_N$ on both sides of~\eqref{e:YN+1} \EEE
	\begin{align}\label{e:YN+1-2}
		Y_{N+1}^{\frac{1}{\gamma_N}}&\le\left(C( \UUU \rho \EEE -r)^{-2\beta}\right)^{\sum_{j=1}^{N}\left(\frac{j}{\gamma_j} \UUU \prod_{k=j}^N(1+\kappa_{k}) \EEE \right)}\left(Y_j+1\right)^{\frac{1}{\gamma_0}\prod_{j=1}^{N}(1+\UUU \kappa_{j} \EEE )}\\
		&\le\left(C(  \UUU \rho \EEE  -r)^{-2\beta}\right)^{\Theta_N\sum_{j=1}^{N}\left(\frac{j}{\gamma_j}\right)}\left(Y_1+1\right)^{\frac{\Theta_N}{\gamma_0}}\,. \nonumber
	\end{align}

	Let us now estimate \UUU $\kappa_j$ as $j\to \infty$. Recalling the definitions~\eqref{eq:gam}, \eqref{eq:gam-2}, \eqref{e:tauj}, \eqref{e:lambdaj}, and~\eqref{e:kappaj}, we \EEE have $\tau_j\to\tau_\infty\in(0,1)$ and
	\begin{equation*}
		1- \UUU \lambda_{j} \EEE =\frac{2q-2p}{2q-p+2 \UUU \sigma_{j} \EEE }\sim\frac{q-p}{2^{j+1}}
	\end{equation*}
	so that
	\begin{equation*}
		\UUU \kappa_{j} \EEE \sim\frac{1-\tau_\infty}{\tau_\infty}\cdot\frac{q-p}{2^{j+1}} \qquad \text{as $j\to\infty$.}
	\end{equation*}
	 Since
	\begin{equation*}
		\log\left(\prod_{j=1}^N(1+ \UUU \kappa_j \EEE)\right)=\sum_{j=1}^N\log(1+  \UUU \kappa_j \EEE )\le\sum_{j=1}^N  \UUU \kappa_j \EEE 
	\end{equation*}
	we have that
	\begin{equation*}
		\lim\limits_{N\to\infty}\Theta_N=\Theta_\infty< +\infty\,.
	\end{equation*}
	\UUU Using \EEE the fact that \UUU $\gamma_j\sim 2^{j+1}$ \EEE we infer that $\sum_{j=1}^\infty\frac{j}{\gamma_j}$ also converges. Therefore,        \UUU we infer from~\eqref{e:YN+1-2} that \EEE
	\begin{equation*}
		\|\cU\|_{L^\infty(B_r)}=\lim\limits_{N\to\infty}\left(\int_{B_r}\cU^{\gamma_N}\,dx\right)^{\frac{1}{\gamma_N}}=\lim\limits_{N\to\infty}Y_{N+1}^{\frac{1}{\gamma_N}}\le \frac{C}{( \UUU \rho \EEE -r)^{\beta'}}\left(\int_{B_{\UUU \rho \EEE }}\cU^{\gamma_0}+1\right)^{\frac{\Theta_\infty}{\gamma_0}}\,.
	\end{equation*}
	\UUU By \EEE definition of $\cU$ and of $\gamma_{0}$ we deduce that
        \begin{equation}
        \label{e:intermediate}
		\|\nabla u\|_{L^\infty(B_r)}\le \frac{C}{(\UUU \rho \EEE -r)^{\beta'}}\left(\int_{B_{\UUU \rho \EEE }}|\nabla u|^{h(2q-p+2)}+1\right)^{\frac{\Theta_\infty}{h(2q-p+2)}}
	\end{equation}
	for some $C,\beta'>0$ depending on $n,p,q,a$.
	
        Let us now improve inequality \eqref{e:intermediate}. \UUU To simplify the notation, we write $K= \frac{2^{*}}{2}$. \EEE Since
	\begin{equation*}
		\GGG p\EEE <h(2q-p+2)< \UUU K \EEE (p+2)\,,
	\end{equation*}
	using interpolation we can write for every $t\le R$
	\begin{equation}\label{eq:gwsbksdk}
		\left(\int_{B_t}|\nabla u|^{h(2q-p+2)}\,dx\right)^{\frac{1}{h(2q-p+2)}}\le\left(\int_{B_t}|\nabla u|^{\GGG p}\,dx\right)^{\frac{\zeta}{\GGG p}}\left(\int_{B_t}|\nabla u|^{\UUU K \EEE (p+2)}\,dx\right)^{\frac{1-\zeta}{\UUU K \EEE (p+2)}}
	\end{equation}
	with the choice
	\begin{equation*}
		\frac{1}{h(2q-p+2)}=\frac{\zeta}{\GGG p}+\frac{1-\zeta}{ \UUU K \EEE (p+2)}\, .
	\end{equation*}
	\GGG We take $\rho=\frac{r+R}{2}$ and $\rho\le s<t\le R$ and 
	 \UUU we rewrite~\eqref{eq:hadh} \EEE with $\sigma= \UUU \sigma_{0} = \EEE 1$ (i.e. $j=0$). 
	 Hence we have
	\begin{equation}\label{e:3000}
		\left(\int_{B_s} |\nabla u|^{\UUU K \EEE (p+2)}\,dx\right)^{\frac{1}{\UUU K \EEE}}\le \frac{C}{(t - s )^2}\left(\int_{B_t}|\nabla u|^{h(2q-p+2)}\,dx+1\right)^{\frac{1}{h}}.
	\end{equation}
	\UUU Plugging~\eqref{e:3000} in \EEE \eqref{eq:gwsbksdk} gives
	\begin{align}\label{e:30001}
		& \left(\int_{B_s}|\nabla u|^{h(2q-p+2)}\,dx\right)^{\frac{1}{h(2q-p+2)}}
        \\
        &
        \qquad \le\frac{C}{(t - s)^{\UUU \upsilon \EEE}}\left(\int_{B_t}|\nabla u|^{\GGG p}\,dx\right)^{\frac{\zeta}{\GGG p}}\left(\left(\int_{B_t}|\nabla u|^{h(2q-p+2)}\,dx\right)^{\frac{1-\zeta}{h(p+2)}}+1\right), \nonumber
	\end{align}
        \UUU where we have set $\upsilon := \frac{2(1 - \zeta)}{p+2}$. \EEE We claim that
	\begin{equation}\label{e:M<1}
		M:=(1-\zeta)\frac{2q-p+2}{p+2}<1.
	\end{equation}
	\UUU In order to prove~\eqref{e:M<1}, let \EEE us denote
	\begin{equation*}
		A:=2q-p+2,\quad B:={\GGG p},\quad \UUU D \EEE :=p+2,\quad \UUU 
	\end{equation*}
	so that
	\begin{equation*}
		\zeta=\frac{\displaystyle\frac{1}{hA}-\frac{1}{\UUU KD \EEE}}{\displaystyle\frac{1}{B}-\frac{1}{\UUU KD \EEE}}
	\end{equation*}
	and
	\begin{equation*}
		M=(1-\zeta)\frac{A}{\UUU D \EEE}=\frac{\frac{1}{B}-\frac{1}{hA}}{\displaystyle\frac{1}{B}-\frac{1}{\UUU KD \EEE}}\cdot\frac{A}{\UUU D \EEE}=\frac{\displaystyle\frac{A}{B}-\frac{1}{h}}{\displaystyle\frac{\UUU D \EEE}{B}-\frac{1}{\UUU K \EEE}}=\frac{\UUU K \EEE}{h}\cdot\frac{hA-B}{\UUU KD \EEE -B}\,.
	\end{equation*}
	\UUU Hence, the condition \EEE $M<1$ is equivalent to
	\begin{equation*}
		\UUU K \EEE(hA-B)<h(\UUU KD \EEE -B) \ \Longleftrightarrow  \ \UUU K \EEE (A- \UUU D \EEE)<\left(\frac{\UUU K \EEE }{h}-1\right)B.
	\end{equation*}
	\UUU Writing the above equivalence \EEE explicitly, we get
	\begin{equation*}
		\frac{n}{n-2}(2q-2p)<\frac{2\alpha}{n-2}{\GGG p} \ \Longleftrightarrow \ n(q-p)<\alpha{\GGG p} \ \Longleftrightarrow \ \frac{q}{p}<1+{\GGG \frac{\alpha}{n}},
	\end{equation*}
	\GGG which is true by assumption, so that \eqref{e:M<1} holds. Using this and applying Young inequality to~\eqref{e:30001} we have:
	\begin{equation}\label{e:improve}
	\begin{split}
		\int_{B_s}|\nabla u|^{h(2q-p+2)}\,dx\le & \ (1-\zeta)\frac{2q-p+2}{p+2}\left(\int_{\UUU B_t \EEE}|\nabla u|^{h(2q-p+2)}\,dx\right)\\
		&+\frac{C}{(\UUU t - s \EEE )^{\UUU \upsilon \omega \EEE}}\left(\int_{ \UUU B_R \EEE}|\nabla u|^{\GGG p}\,dx\right)^{\frac{\zeta\UUU h (2q - p + 2) \omega \EEE}{\GGG p}} \\
		&+\frac{C}{( \UUU t - s \EEE )^\upsilon}\left(\int_{ \UUU B_R \EEE}|\nabla u|^{\GGG p}\,dx\right)^{\frac{\zeta h(2q-p+2)}{\GGG p}} 
	\end{split}
	\end{equation}
	with $\omega =\left(\frac{p+2}{(1-\zeta)(2q-p+2)}\right)'=\frac{1}{1-M}$. Lemma \ref{lem:est}, \eqref{e:improve}, and~\eqref{e:intermediate}  finally yield \EEE
	\begin{equation*}
		\|\nabla u\|_{L^\infty(B_r)}\le \frac{C}{(R-r)^{\mu}}\left(\int_{B_R}|\nabla u|^{p}+1\right)^{\frac{\Theta}{p}}
	\end{equation*}
	for some $C,\mu>0$, and $\Theta\ge 1$ depending on $n,p,q,a$. This concludes the proof of~\eqref{eq:lip}.
\end{proof}
\subsection{Main proof}
We now proceed to prove the main theorem.
\begin{proof}[Proof of Theorem \ref{thm:main}]
	Proposition \ref{prop:lip} gives
	\begin{equation}\label{eq:mainest}
		\|\nabla u^\varepsilon\|_{L^\infty(B_{R/3})}\le C \left(\int_{B_R}|\nabla u^\varepsilon|^p+1\right)^{\Theta/p}
	\end{equation}
	for every ball $B_R$ with radius $R\le1$, where $C$ depends on $n,p,q,a,R$. Thanks to Proposition~\ref{prop:approx} $u^\varepsilon\to U$ strongly in $W^{1,p}(B_R)$ and thus, passing to the limit in \eqref{eq:mainest}, we get 
	\begin{equation*}
		\|\nabla u\|_{L^\infty(B_{R/3})}\le C \left(\int_{B_R}|\nabla u|^p+1\right)^{\Theta/p}\,.
	\end{equation*}
    This concludes the proof of the theorem.
\end{proof}

\subsection*{Acknowledgements}

The work of the authors was funded by the University of Naples Fe\-derico II through FRA Project ``ReSinApas". S.A. was also supported by the FWF Austrian Science Fund through the Project 10.55776/P35359, by the PRIN 2022 Project 2022HKBF5C ``Variational Analysis of Complex Systems in Materials Science, Physics and Biology'', and by Gruppo Nazionale per l'Analisi Matematica, la Probabilit\`a e le loro Applicazioni (GNAMPA-INdAM, Project 2025: DISCOVERIES - Difetti e Interfacce in Sistemi Continui: un'Ottica Variazionale in Elasticit\`a con Risultati Innovativi ed Efficaci Sviluppi). The research of C.L. was supported by the PRIN 2022 Project 2022E9CF89 “Geometric Evolution Problems and Shape Optimizations”. The authors are members of Gruppo Nazionale per l'Analisi Matematica, la Probabilit\`a e le loro Applicazioni (GNAMPA-INdAM).

\bibliographystyle{siam}
\bibliography{bibODM_ALM}
\end{document}